\documentclass[psamsfonts, reqno]{amsart}
\usepackage{amsfonts}
\usepackage{amsthm}
\usepackage{amsmath}
\usepackage{amssymb}
\usepackage{lipsum}
\usepackage{mathrsfs}
\usepackage{esint}
\usepackage{wrapfig}
\usepackage{dirtytalk}
\usepackage{tabu}
\usepackage{mathtools}
\makeatletter
\newtheorem{thm}{Theorem}[section]
\newtheorem{lem}{Lemma}[section]
\newtheorem{cor}{Corollary}[section]
\newtheorem{prop}{Proposition}[section]
\theoremstyle{remark}
\newtheorem*{remark}{Remark}

\newcommand{\al}[2]{\alpha_{#1}^{(#2)}}
\newcommand{\floor}[1]{\left\lfloor #1 \right\rfloor}

\usepackage[utf8]{inputenc}
\usepackage{graphicx}
\numberwithin{equation}{section}
\title{New Estimates on the Bounds of Brunel's Operator}
\author{I. Assani$^*$}
\address{$^*$ Department of Mathematics, University of North Carolina at Chapel Hill}
\email{assani@email.unc.edu}
\urladdr{https://idrisassani.web.unc.edu/}

\author{R.S. Hallyburton$^{**}$}
\address{$^{**}$ University of North Carolina at Chapel Hill}
\email{scotch@live.unc.edu}

\author{S. McMahon$^\dagger$}
\address{$^\dagger$ University of North Carolina at Chapel Hill}
\email{seanmcmahon1394@gmail.com}

\author{S. Schmidt$^\ddag$}
\address{$^\ddag$ University of North Carolina at Chapel Hill}
\email{sps@live.unc.edu}

\author{C. Schoone$^\mathsection$}
\address{$^\mathsection$ University of North Carolina at Chapel Hill}
\email{csch2@live.unc.edu}

\date{\today}

\begin{document}
\maketitle

\let\thefootnote\relax\footnotetext{R.S.H. and C.S. acknowledge support from the Tom and Elizabeth Long Excellence Fund for Honors administered by Honors Carolina.}
\begin{abstract}
We study the coefficients of the Taylor series expansion of powers of the function $\psi(x)=\frac{1-\sqrt{1-x}}{x}$, where the Brunel operator $A\equiv A(T)$ is defined as $\psi(T)$ for any mean-bounded $T$. We prove several new precise estimates regarding the Taylor coefficients of $\psi^n$ for $n\in\mathbb{N}$. We apply these estimates to give an elementary proof that for any mean-bounded, not necessarily positive operator $T$ on a Banach space $X$, the Brunel operator $A(T):X\to X$ is power-bounded and satisfies $\sup_{n\in\mathbb{N}} \|n(A^n-A^{n+1})\| < \infty$ (equivalently, $A(T)$ is a Ritt operator). Along the way we provide specific details of results announced by A. Brunel and R. Emilion in \cite{Brunel}.
\end{abstract}

%%INTRODUCTION%% 
\section{Introduction}
In \cite{Brunel}, A. Brunel and R. Emilion introduced properties of the Brunel operator.
If we write $\psi^n(x)= x^{-n}\left(1-\sqrt{1-x}\right)^n=\sum_{p=0}^{\infty} \alpha_p^{(n)}
x^p$ and $T$ is an operator from a Banach space $X$ to itself, then the Brunel operator is defined as
$A(T) \equiv \psi(T)= \sum_{p=0}^{\infty} \alpha_p^{(1)}T^p$. We first recall a few definitions. An operator $T: X \to X$ is a \emph{contraction} if $\Vert T\Vert\leq1$, and \emph{power-bounded} if $\sup_{p\in\mathbb{N}}\Vert T^p\Vert<\infty$. $T$ is said to be \emph{mean-bounded} if the Cesàro
averages $M_N(T)=\frac{1}{N}\sum_{n=0}^{N-1}T^n$ are uniformly bounded, i.e. $\sup_{N\in\mathbb{N}}\|M_N(T)\| < \infty$. If a partial ordering is defined on $X$, we say that $T$ is \emph{positive} if for $X^+=\{f \in X, f \geq 0\}$, we have $TX^+\subseteq X^+$. Lastly, a \emph{Ritt operator} $S$ is a power-bounded operator on a Banach space which satisfies the \emph{Ritt condition}:
\[\sup_{n\in\mathbb{N}}n\left\Vert S^{n+1}-S^n\right\Vert<\infty.\] 

% For $f\in L^p(X,\mu)$, an operator $T:L^p(X,\mu)\to L^p(X,\mu)$ is called \textit{positive} when $f\geq0$ implies $Tf\geq0$.

The Brunel operator $A(T)$ is well-defined if $T$ is a contraction. More generally, if $T$ is mean-bounded, then $A$ is also well-defined (see Theorem \ref{mean_bound_bound}). In ergodic theory, typical examples of
contractions on $L^p(X,\mu)$ are given by measure-preserving transformations. When $S$ is measure-preserving, the
operator $T$ defined by the relation $Tf = f\circ S$ is an isometry in each $L^p$ for $1\leq p \leq \infty$; this makes the Brunel operator associated with $T$ a contraction in all $L^p$ spaces with $1\leq p\leq\infty$. The classical Birkhoff pointwise and Von Neumann norm convergence theorems say that if $f \in L^p$
then the Cesàro averages $M_N(T)f$ converge $\mu$-almost everywhere and in $L^p$ norm for $1\leq
p<\infty$. One of the major properties of the Brunel operator is that these two classical results are
translated into convergence of powers of $A$. In fact, when $T$ is positive, the convergence (in norm or pointwise) of these
averages is equivalent to the convergence (norm or pointwise) of the powers $A^nf$. More on this equivalence is given in \cite{preprint}.

In ergodic theory, one typically considers the norm or pointwise convergence of the sums $\frac{1}{N}\sum_{n=0}^{N-1} T^nf$, where $T$ is an operator on the space $L^p(X,\mu)$ for some $p\in[1,\infty)$ (or $p=\infty$ when $\mu(X)<\infty)$, and where $f\in L^p(X,\mu)$. If this sum converges in norm as $N\to\infty$, then $\sup_{N\in\mathbb{N}}\left\Vert\frac{1}{N}\sum_{n=0}^{N-1} T^nf\right\Vert<\infty$. It then follows from the Banach-Steinhaus theorem that $T$ is mean-bounded; therefore, the condition of mean-boundedness is the most general condition under which one can expect the convergence of the Cesàro
averages.

% Although it is a special case of being mean-bounded, we consider the case where $T$ is power-bounded separately since the techniques are informative for the more general case. 

In this paper, we develop  properties of $\psi^n(x)$ and the resulting Taylor coefficients $\alpha_{p}^{(n)}$ in order to prove general results about the Brunel operator. After some preliminary results, Section 3 establishes properties regarding how these coefficients interact with each other. In particular, we prove a bound on the difference $\left|\alpha_p^{(n+1)}-\alpha_p^{(n)}\right|$ that is independent of $n$. Lemmas \ref{ijk_lemma} and \ref{Upper_Lower_Kn} establish monotonicity results of the coefficients which are fundamental to the development of the summation estimates established in Section 4. Using elementary techniques, we determine the value of many constants that bound these summations, including the following estimate (see Theorem \ref{main}):
\[\sup_{n\in\mathbb{N}}n\sum_{p=0}^\infty\left|\al{p}{n}-\al{p}{n+1}\right|\leq\frac{33}{2}.\]
These coefficient summations in turn allow for simple proofs of the theorems in Section 5 in which we develop fundamental properties of the Brunel operator. 

%nd derive a closed form for the sum $\sum_{n=1}^N\alpha_p^{(n)}$.

In \cite{Brunel}, Brunel and Emilion announced, for $T$ mean-bounded on a Banach Space $X$, that
$\|A(T)^n- A(T)^{n+1}\| \rightarrow_n 0 $ (2.7, pg. 105). In \cite{Lootgieter}, Lootgieter proves that, for $T$ a positive mean-bounded operator,
$A(T)$ (and, in general any positive power-bounded power-subadditive operator) is a Ritt operator on $L^p$ (Proposition (**), pg. 4305). In Theorem \ref{mean_bound_bound} we drop the positivity condition and prove that, for any Banach space $X$, $A(T)$ is a Ritt operator on $X$ when $T$ is mean-bounded. Results based on the study of the Taylor coefficients of $\psi$ play a central role in this theorem. 

Recent work in \cite{Dungey} shows that for any power-bounded operator $S$ and any $r\in(0,1)$, the operator $I-(I-S)^r$ is a Ritt operator. The Brunel operator is instead generated via the functional calculus by $\psi(x)=x^{-1}(1-\sqrt{1-x})$, and is a Ritt operator whenever $T$ is mean-bounded, a strictly weaker condition than power-boundedness. Hence for any mean-bounded $T$, the powers $A^n(T)$ of the Brunel operator not only converge in norm as $n\to\infty$ but even satisfy the Ritt condition $\sup_{n}n\left\Vert A^n(T)-A^{n+1}(T)\right\Vert<\infty$.

%Operators of the form $I-(I-T)^r$ for $T$ power-bounded and $r\in(0,1)$ were shown to be Ritt operators in \cite{Dungey} (Theorem 4.3). In this paper, we extend these results to study operators that arise from the function $\psi(x)=(1-\sqrt{1-x})/x$. We show that for the weaker condition of $T$ mean-bounded, we obtain an analogous result to Theorem 4.3 in the case $r=1/2$; that is, that $\psi(T)$ is a Ritt operator.
Our methods use elementary techniques; specifically, we avoid calculus of residues as used by Brunel in \cite{Early_Brunel}. Along the way we also sharpen and give specific details of
other properties announced in \cite{Brunel}, including 2.3 and 2.6 (pg. 105) (see Corollary \ref{monotone_positive_operator_cor}, and Proposition \ref{inversion_formula} \& Theorem \ref{mean_bound_bound}, respectively). The computations made in the paper are often delicate, and our aim is to provide as many details as reasonably possible to make verification easy for the interested reader. We give detailed proofs of main results and technical lemmas in the main paper; auxiliary results which are lengthy but otherwise easily verified are placed in the appendix for the sake of completeness. As a general note, the constants derived in the following estimates are not necessarily the minimal constants.

\section{Preliminaries}

A preliminary version of Brunel's operator appeared in \cite{Early_Brunel} (pg. 335), inspired by a continuous semi-group from Dunford and Schwartz in \cite{Dunford-Schwartz}. Later, some results from this paper were presented in more detail in \cite{Krengel}. We begin by restating and partially reproving one of these results from \cite{Krengel} (pgs. 212-213). A modification of the function $\xi$ as defined below yields the function $\psi$ which defines the Brunel operator.
%KRENGELS PROOF%oh 
\begin{lem}\label{krengel}
For $\xi(x) = 1-\sqrt{1-x}$ and $n\in\mathbb{N}$, the coefficients $\beta_p^{(n)}$ in the expansion
\begin{equation}{\label{krengelsum}}
   [\xi(x)]^n = \sum_{p=0}^{\infty} \beta_p^{(n)} x^p 
\end{equation}
are given by:
\begin{equation}{\label{alphadef}}
\beta_p^{(n)} =
   \begin{cases}
   0  &  p<n; \\
   \frac{n}{2p} 2^{n+1-2p}{{2p-n-1}\choose{p-1}} & p\geq n .
   \end{cases}
\end{equation}
\end{lem}
\begin{proof}
We prove the lemma by induction. We define:
\[\binom{1/2}{k}=\frac{(1/2)(1/2-1)(1/2-2)\cdots\big(1/2-(k-1)\big)}{k!}.\]
From this definition it follows that:
\begin{equation*}
    {{1/2}\choose{p}} = {{2p}\choose{p}} \frac{(-1)^{p+1}}{2^{2p}(2p-1)}.
\end{equation*}
Additionally, for $x\in[-1,1]$:
\begin{equation*}
    (1-x)^{1/2} = \sum_{p=0}^{\infty}(-1)^p{{1/2}\choose{p}}x^p.
\end{equation*}
Consider the base case $n=1$. We have:
\begin{align*}
\begin{split}
 \xi(x) &= 1 - \sum_{p=0}^{\infty}(-1)^p{{1/2}\choose{p}}x^p \\
 &= \sum_{p=1}^{\infty} {(-1)^{p+1}}{{1/2}\choose{p}}x^p \\
&= \sum_{p=1}^{\infty}{(-1)^{p+1}} {{2p}\choose{p}} \frac{(-1)^{p+1}}{2^{2p}(2p-1)}x^p.
    \end{split}       
\end{align*}
From this we can derive that: \begin{equation}{\label{krengelder}}
    \beta_p^{(1)} = \frac{1}{2p}2^{2-2p}{{2p-2}\choose{p-1}}.
\end{equation}
when $p\geq{1}$ and $0$ otherwise. This completes the base case. Equation (\ref{krengelder}) shows that as we multiply $\xi(x)$ by itself $n$ times (which makes sense since the series converges absolutely for all $x\in [-1,1]$), every coefficient in front of $x^p$ for $p<n$ must be 0. The proof of the induction step for $p\geq k$, which we omit here for brevity (see \cite{Krengel}, pgs. 212-213), uses the following recursion for $k\geq1$ that Krengel attributes to N. Neumann (pg. 212):
\begin{equation}\label{eq:krengel_recurrence}
    [\xi(x)]^{k+1} = [\xi(x)]^{k-1}[\xi(x)]^2 = 
[\xi(x)][2\xi(x)-x]^{k-1}.\qedhere
\end{equation}
\end{proof}
%%END KRENGEL %%

We can now use Krengel's result to establish some basic properties of $\psi(x)$.

\begin{lem}\label{applied_krengel}
For the function $\psi (x) = \frac{1- \sqrt{1-x}}{x}$, the coefficients $\alpha_p^{(n)}$, $n \in \mathbb{N}$ in the expansion $$[\psi(x)]^n = \sum_{p=0}^{\infty} \alpha_p^{(n)} x^p $$ are given by:
\begin{equation}{\label{maindefintion}}
    \alpha_p^{(n)} = \begin{cases}
    1, & n =p = 0;\\
    \frac{n}{n+p}2^{-n-2p}{{n+2p-1}\choose{p}}, & \emph{otherwise}.
    \end{cases}
\end{equation}
For each $n\in\mathbb{N}$, the series converges for all $x\in[-1,1]$.
\end{lem}
\begin{proof}
For $x\neq0$ it follows immediately from Lemma \ref{krengel} that each $\alpha_p^{(n)}$ satisfies the given expression by dividing by $x$ on both sides of \eqref{krengelsum} and shifting the $p$ indices in the subsequent series. Note that, when $n \neq 0$, $\alpha_p^{(n)}$ is nonzero for all $p$ due to shifting. For the case of $x = 0$, observe that:
\begin{equation*}
    \lim_{x \to 0} \bigg{(}\frac{[\xi(x)]}{x}\bigg{)}^n = \frac{1}{2^n} = \alpha_0^{(n)}.
\end{equation*}
Further, we have by a brief computation:
\[\left|\frac{\alpha_{p+1}^{(n)}}{\alpha_p^{(n)}}\right|=\frac{(n+2p+1)(n+2p)}{4(p+1)(n+p+1)},\]
so the radius of convergence of the series is $1$ by the ratio test. Moreover, at $x=1$ and $x=-1$, the series expansion of $\psi^n(x)$ coincides up to a sign with that of $\xi^n(x)$, which converges absolutely for all $x\in[-1,1]$. Hence, the series expansion of $\psi^n(x)$ is convergent for all $x\in[-1,1]$.
\end{proof}
\begin{remark}
    It is often helpful to rewrite $\alpha_p^{(n)}$ in the following more symmetric form:
    \[\alpha_p^{(n)}=\frac{n}{n+2p}2^{-(n+2p)}\binom{n+2p}{p}.\]
\end{remark}
The following result gives a useful recurrence relation for the coefficients $\alpha_p^{(n)}$ for $n\geq1$:
\begin{lem}\label{alpha_reurrence}
    The coefficients $\alpha_p^{(n)}$ satisfy the following recurrence:
    \begin{equation}\label{eq:alpha_recurrence}
        \alpha_p^{(n+1)}=2\alpha_{p+1}^{(n)}-\alpha_{p+1}^{(n-1)}.
    \end{equation}
\end{lem}
\begin{proof}
    We use the recurrence in Equation \eqref{eq:krengel_recurrence} to obtain:
    \[\psi^{n+1}(x)=\frac{\xi^{n+1}(x)}{x^{n+1}}=\frac{2}{x}\psi^n(x)-\frac{1}{x}\psi^{n-1}(x).\]
    By comparing coefficients in the Taylor series expansion of $\psi^n$ and shifting the $p$-indices, the desired recurrence follows.
\end{proof}

We now have a theorem regarding the interactions of powers of the coefficients. Part $(2)$ (the \emph{power-subadditivity} of $\alpha_p^{(n)}$) of the following was proved in Example 1.5 of \cite{Lootgieter}. Here we obtain the result as a consequence of $(1)$. 

\begin{thm}\label{monotonicity_results}
The following are properties of $\alpha_p^{(n)}$ for all $n,m,p \in \mathbb{N}$:
\begin{enumerate}
    \item   $\frac{\alpha_p^{(n+1)}}{n+1} \leq \frac{\alpha_p^{(n)}}{n}; $ 
    \item $\alpha_p^{(m+n)} \leq \alpha_p^{(m)}+ \alpha_p^{(n)}. $ 
\end{enumerate}

\end{thm}

\begin{proof}

For (1), we use Lemma \ref{applied_krengel} to get the following chain of equivalences:

\begin{align*}
&\frac{\alpha_p^{(n+1)}}{n+1}-\frac{\alpha_p^{(n)}}{n} &\leq 0 \\
    \frac{n(n+1)}{n+p+1}2^{-1-n-2p} {{n+2p}\choose{p}}& -\frac{n^2}{n+p}2^{-n-2p}{{n+2p-1}\choose{p}} \\
     &- \frac{n}{n+p}2^{-n-2p}{{n+2p-1}\choose{p}} &\leq 0\\
     &n2^{-n-2p}{{n+2p}\choose{p}} \bigg{(} \frac{n+1}{2(n+p+1)} - \frac{n+1}{n+2p}\bigg{)} &\leq 0\\ 
     &-\frac{n^2+3n+2}{2(n+p+1)(n+2p) } &\leq 0
\end{align*}
 where the last line is true for all positive $n$ and $p$. Thus, the result holds. \vspace{1 mm}

For (2), fix $m,n\in \mathbb{N}$ and without loss of generality suppose that $m\geq n$. By (1), we have that $\frac{\alpha_p^{(k)}}{k}$ is decreasing for all $k$. Thus:
\begin{equation*}
    m\alpha_p^{(n)} \leq n\alpha_p^{(m)}
\end{equation*}
and therefore:
\begin{equation*}
    \frac{\alpha_p^{(m+n)}}{m+n} \leq \frac{\alpha_p^{(n)}}{n}.
\end{equation*}
This implies that:
\begin{equation*}
    \alpha_p^{(m+n)} \leq \frac{(m+n)\alpha_p^{(n)}}{n} \leq \alpha_p^{(n)} + \alpha_p^{(m)},
\end{equation*}
which proves (2).
\end{proof}
\begin{cor}\label{montonicity_results_corollary}
$\psi(x)$  satisfies the enumerated properties in Theorem \ref{monotonicity_results}, i.e. the following are true for all $n,m \in \mathbb{N}$, $x \in [0,1]$:
\begin{enumerate}
    \item   $\frac{\psi^{n+1}(x)}{n+1} \leq \frac{\psi^n(x)}{n} $;
    
    \item $\psi^{m+n}(x) \leq \psi^m(x)+ \psi^n(x) $;
    
    \item $\big(1-\psi^n(x)\big)\big(1-\psi^m(x)\big)\leq1.$
\end{enumerate}
\end{cor}
\begin{proof}
Since: \begin{equation*}
    [\psi(x)]^n = \sum_{p=0}^\infty \alpha_p^{(n)} x^n
\end{equation*}
for all $x \in [0,1]$, (1) and (2) are consequences of Theorem 1.1, since $\alpha_p^{(n)}$ is positive. (3) follows from (2), as:
\[\big(1-\psi^n(x)\big)\big(1-\psi^m(x)\big)=1-\psi^n(x)-\psi^m(x)+\psi^{m+n}(x)\leq1.\qedhere\]
\end{proof}
\begin{cor}\label{monotone_positive_operator_cor}
Let $T$ be a positive operator on a Banach lattice $X$. Then whenever $A = \psi(T)$ is well-defined, the following are true for all $m,n \in \mathbb{N}$: \begin{enumerate}
    \item $\frac{A^{n+ 1}}{n+1} \leq \frac{A^n}{n}$;
    \item $A^{m + n} \leq A^m + A^n$;
    \item $\big(I-A^m\big)\big(I-A^n\big)\leq I$.
\end{enumerate}
\end{cor}
\begin{remark}
    This result was announced in \cite{Brunel} without details (pg. 105, Theorem 2.6). Property (2) was proved in \cite{Lootgieter} (pg. 4301).
\end{remark}
%%END INTRODUCTION%%

%%ESTIMATE%%
\section{Basic Estimates}
In this section we prove several basic estimates on the Taylor series coefficients $\alpha_p^{(n)}$ which will be applied later. In particular, we derive sharp bounds on the difference of coefficients $\left|\alpha_p^{(n+1)}-\alpha_p^{(n)}\right|$, and subsequently for $\left|{\psi(x)}^{n+1} - {\psi(x)}^n\right|$. We first state a proposition proved in \cite{Hirschhorn}; we will often use this estimate in the sequel. 
\begin{prop}\label{Hirschhorn_bound}
For all $p \in \mathbb{N}^+$: \begin{equation}
    \frac{1}{\sqrt{\pi p}}\left(1-\frac{1}{4p}\right) \leq 4^{-p}{2p \choose p } \leq \frac{1}{\sqrt{\pi p}}.
\end{equation}
\end{prop}
We will often use the slightly weaker estimate of \begin{equation}
    \frac{1}{\sqrt{\pi(p+1)}}\leq 4^{-p}\binom{2p}{p},
\end{equation}
which holds when $p = 0$ (by convention taking $0!=1$). From this we derive the following estimate. 

\begin{prop}\label{alpha_estimate}
    For all $n, p \in \mathbb{N}, p \geq 1$:\begin{equation}
        \alpha_{p}^{(n)} \leq \frac{n}{n + 2p}\frac{1}{\sqrt{\pi p}}.
    \end{equation}
\end{prop}
\begin{proof}
    The result follows from: \begin{equation}
        \alpha_p^{(n)} = \frac{n}{n + 2p}4^{-p}\binom{2p}{p}\prod_{j=0}^{n-1}\frac{2p + n -j}{2(p + n-j)} \leq \frac{n}{n+2p}\frac{1}{\sqrt{\pi p}}.\qedhere
    \end{equation}
\end{proof}
\begin{cor}\label{N_alpha_N_estimate}
    For $n \in \mathbb{N}$ fixed, \begin{equation}
        \lim_{N \to \infty}N\alpha_N^{(n)} = 0. 
    \end{equation}
\end{cor}
We now derive a few properties about the monotonicity of various sections of $\alpha_ p^{(n)}$.

\begin{lem}\label{ijk_lemma}
Let $\alpha_p^{(n)}$ be as in Lemma \ref{applied_krengel}. Then the following properties hold for all $n,p\in\mathbb{N}^+$:
\begin{enumerate}

    \item $\alpha_p^{(n)} \leq \alpha_{p+1}^{(n)}$ if and only if $p \leq \frac{n^2-3n-4}{6} = I_n$;
    
    \item $\alpha_p^{(n)} \geq \alpha_{p}^{(n+1)}$ if and only if $p \leq \frac{n^2+n}{2} = J_n$;
    
    \item Suppose $n>1$. If $p \leq \frac{n^4-2n^3-13n^2-10n}{12n^2+12n-24} = K_n$, then $\alpha_p^{(n)}-\alpha_p^{(n+1)} \leq \alpha_{p+1}^{(n)}-\alpha_{p+1}^{(n+1)}$;
    
    \item Under the hypothesis of \text{(3)}, $|\alpha_p^{(n)}-\alpha_p^{(n+1)}| \leq |\alpha_{p+1}^{(n)}-\alpha_{p+1}^{(n+1)}|$.
\end{enumerate}

\end{lem}

To illustrate the significance of the above sections, consider Figure 1 describing $K_n$. Since $K_n$ increases quadratically as $n$ increases, we will be able to split $\sum_{p=0}^\infty |\alpha_p^{(n+1)}-\alpha_p^{(n)} |$ into simpler parts. The sum after $\lfloor K_n \rfloor$, as we take $n$ large, will become less significant, as will be seen in Theorem \ref{sum_coefficient_bound}.

\begin{figure}[htp]
    \centering
    \includegraphics[width=10cm]{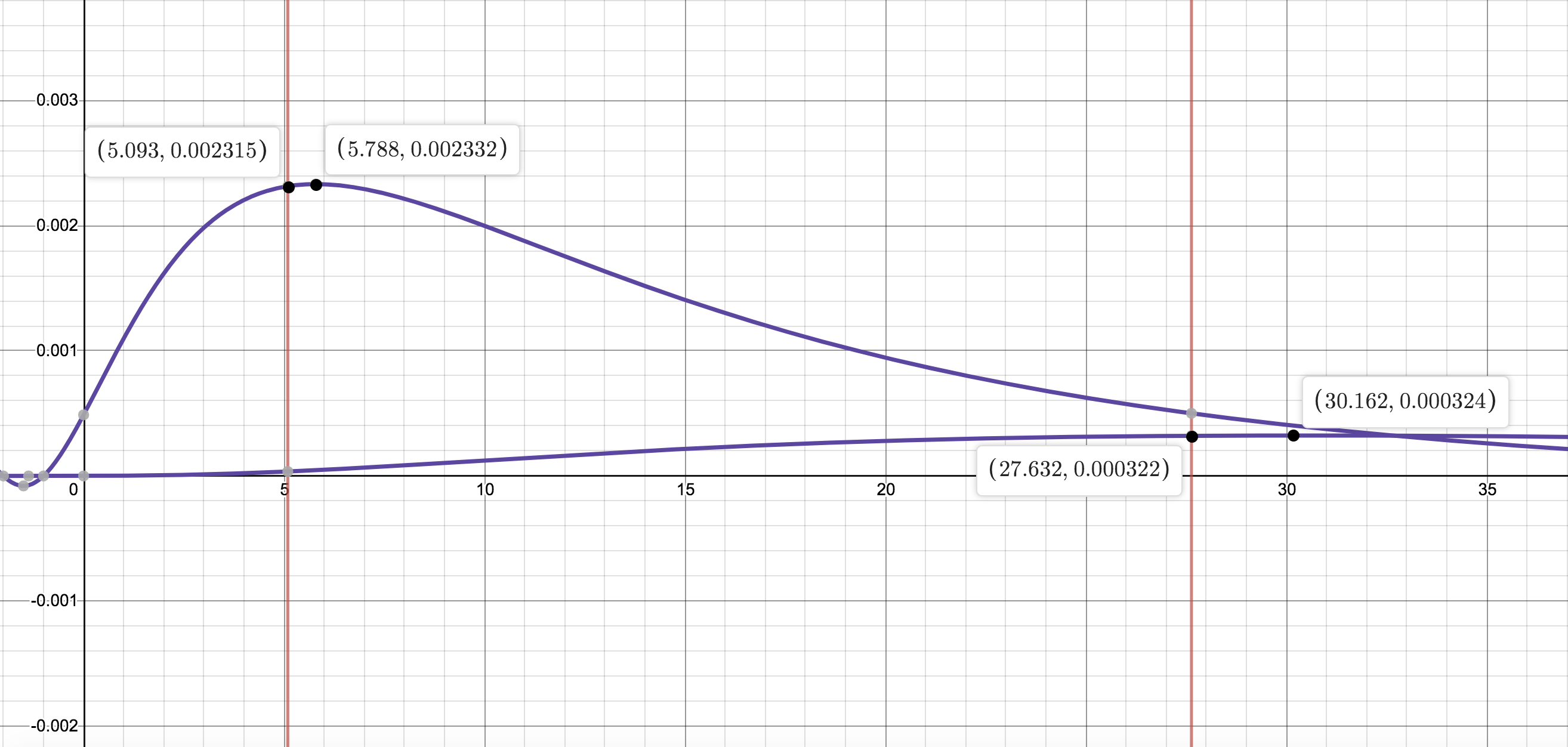}
    \caption{Graph of ($p$, $\alpha_p^{(n)} - \alpha_p^{(n+1)}$) (purple curves) along with $K_n$ (red lines) for $n=10, 20$. Note that $K_n$ is less than the $p$ such that $\alpha_p^{(n)} - \alpha_p^{(n+1)}$ attains its local maximum value. Graphed with Desmos.}
    \label{fig:K_n}
\end{figure}

\begin{proof}
For (1), we have the following chain of equivalences. 
\begin{align}
    \alpha_p^{(n)} &\leq \alpha_{p+1}^{(n)}\label{eq:(1)initial}\\
    \begin{split}
     \frac{n}{n+p}2^{-n-2p}{{n+2p-1}\choose{p}} &\leq \frac{n}{n+p+1}2^{-n-2p-2}{{n+2p+1}\choose{p+1}} \\
     \frac{4(n+p+1)}{n+p} &\leq \frac{(n+2p+1)(n+2p)}{(p+1)(n+p)}\\
     4pn+4p^2+ 4p+ 4n + 4p + 4 &\leq n^2+2pn + 2pn + 4p^2+n + 2p
     \end{split}\\
     \frac{n^2-3n -4}{6}&\leq p.\label{eq:(1)final}
\end{align}
Thus (\ref{eq:(1)initial}) is equivalent to (\ref{eq:(1)final}) and this proves (1). For (2), we again have a chain of equivalences:
\begin{align*}
   \alpha_p^{(n)}-\alpha_p^{(n+1)} &\geq 0\\ 
     \bigg{[}2^{-n-2p}{{n+2p}\choose{p}}\bigg{]}\bigg{(}{ \frac{n}{n+p}{\frac{n+p}{n+2p}}-\frac{(n+1)}{2(n+p+1)}}\bigg{)} &\geq 0\\
     2n^2 +2pn+2n - n^2 -2np -n -2p &\geq 0\\
     \frac{n^2 + n}{2}&\geq p.
\end{align*} For (3), see the appendix to see that the following two inequalities are equivalent:

\begin{align}
     \alpha_p^{(n)} - \alpha_p^{(n+1)} &\leq \alpha_{p+1}^{(n)}-\alpha_{p+1}^{(n+1)}\label{eq:KN_intial}\\  
    p(12n^2+12n-12)-12p^2 &\leq n^4-2n^3-13n^2-10n.
\end{align}
Thus:
\begin{equation*}
    p \leq \frac{n^4-2n^3-13n^2-10n}{12n^2+12n-24}
\end{equation*}
implies (\ref{eq:KN_intial}), and this proves (3). For (4), notice that $K_n \leq J_n$ for all $n\in \mathbb{N}$. Thus,
\begin{align*}
    \left|\alpha_p^{(n+1)}-\alpha_p^{(n)} \right| &= \alpha_p^{(n)}-\alpha_p^{(n+1)}\\ &\leq 
    \alpha_{p+1}^{(n)}-\alpha_{p+1}^{(n+1)}\\ &\leq 
   \left| \alpha_{p+1}^{(n+1)}-\alpha_{p+1}^{(n)}\right|.
\end{align*}
This completes the proof. 
\end{proof}
\begin{remark}
    In the appendix, we construct a stronger version of $K_n$ to be used in Theorem \ref{mean_bound_bound}. The $K_n$ above is sufficient for our purposes at the moment, however.
\end{remark}
These properties of monotonicity will allow us to estimate $\left|\alpha_p^{(n+1)}-\alpha_p^{(n)}\right|$ with sharper precision. Before this, we present a result that will appear in the proof of Theorem \ref{cesaro_square_root_average_bound}. The technique used in this proof motivates an estimate used in Proposition \ref{uniform_p_estimate}. We note here that $[x]$ denotes the integer part of $x + 1$; this is the convention used in \cite{Brunel}.
\begin{prop}\label{product_estimate}
For all $p, n \in \mathbb{N}^+, 1 \leq p \leq n$: \begin{equation}
   P_p^{(n)} \equiv 2^{-([ \sqrt{n}] -1)}\prod_{k = 0}^{[ \sqrt{n}]-2}\frac{2p + [ \sqrt{n}] -1 -k}{p + [ \sqrt{n}] -1 -k}\leq e^{-1/10}.
\end{equation}
\begin{proof}
Note that each component of the product is increasing in $p$. This follows from the following computation:
\[\frac{2p+[\sqrt n]-1-k}{p+[\sqrt n]-1-k}=2-\frac{[\sqrt n]-1-k}{p+[\sqrt n]-1-k},\]
and the term on the right is decreasing in $p$. Since every component of the product is positive, it follows that $P_p^{(n)} \leq P_{n}^{(n)}$ for all $p\leq n$. We have: \begin{align*}
    P_n^{(n)} &= \prod_{k = 0}^{[ \sqrt{n}] -2} \left(1-\frac{{[ \sqrt{n} ] - 1-k}}{2(n + [ \sqrt{n}] -1 -k)}\right)\\
    &\leq \prod_{k = 0}^{[ \sqrt{n}] -2} \left(1-\frac{{[ \sqrt{n} ] - 1-k}}{2(n + [ \sqrt{n}] -1 )}\right)\\
    &= \prod_{k = 0}^{[ \sqrt{n}] -2} \left(1-\frac{{k + 1}}{2(n + [ \sqrt{n}] -1 )}\right).\\
\end{align*}
The last equality follows from reversing the order of terms in the product. 
For $a, b \in \mathbb{R}^+$, we note that: \[(1 - a)(1-b) \leq \left(1 - \frac{a+b}{2}\right)^2.\]
Hence: \begin{align*}
    P_n^{(n)} &\leq \left(1 - \frac{\frac{[ \sqrt{n}]}{2}}{2(n + [\sqrt{n}]-1)}\right)^{[ \sqrt{n}] -1}\\
    &=\left(1 - \frac{[ \sqrt{n}]}{4(n + [\sqrt{n}]-1)}\right)^{[ \sqrt{n}] -1}.
\end{align*}
Then: \begin{align*}
    \log(P_n^{(n)}) &\leq \left([ \sqrt{n}]-1\right)\log\left(1 - \frac{[ \sqrt{n}]}{4(n + [\sqrt{n}]-1)}\right) \\
    &\leq -\frac{[ \sqrt{n}] ^2 - [ \sqrt{n}]}{4(n + [\sqrt{n}]-1)}\\
    & \leq - \frac{[\sqrt{n}]^2- [\sqrt{n}]}{4([\sqrt{n}]^2 + [\sqrt{n}]-1)}\\
    &\leq - \frac{1}{10}.
\end{align*}
This last line follows since the fraction is decreasing in $n$ and $[1] = 2$. Hence $P_p^{(n)} \leq e^{-1/10}< 1$ for all $n, p$, $1\leq p\leq n$. 
\end{proof}
\end{prop}
\begin{prop}\label{uniform_p_estimate}
    The following bound holds for all $n,p\in\mathbb{N}$:
    \begin{equation}\label{eq:prop3_bound}
        \left|\alpha_p^{(n+1)}-\alpha_p^{(n)}\right|\leq\alpha_p^{(1)}\leq\frac{1}{2(1+p)}\sqrt\frac{1}{\pi p}.
    \end{equation}
\end{prop}
\begin{proof}
    The result is trivial if $p\geq J_n$, since in this case we have:
    \[\left|\al{p}{n+1}-\al{p}{n}\right|=\al{p}{n+1}-\al{p}{n}\leq\al{p}{n}+\al{p}{1}-\al{p}{n}=\al{p}{1},\]
    using the subadditivity of the coefficients $\al{p}{n}$. For $p\leq J_n$, consider the quotient $\left(\al{p}{n}-\al{p}{n+1}\right)/\al{p}{1}$; the fraction is necessarily positive, since $p\leq J_n$. We claim that this quotient is bounded uniformly in $n$ and in $p$ by $1$, which suffices to prove the claim. This is easily checked for $1\leq n\leq 4$, as there are only finitely many values of $p$ to test, so for the remainder of the proof suppose that $n\geq 5$. Substituting the definition of $\al{p}{n}$ and simplifying, we obtain:
    \begin{align*}
        \frac{\al{p}{n}-\al{p}{n+1}}{\al{p}{1}}&=\frac{\frac{n}{n+2p}2^{-(n+2p)}\binom{n+2p}{p}-\frac{n+1}{n+2p+1}2^{-(n+2p+1)}\binom{n+2p+1}{p}}{\frac{1}{1+2p}2^{-(1+2p)}\binom{1+2p}{p}} \\
        &=\frac{2^{-(n+2p)}\binom{n+2p}{p}\left[\frac{n}{n+2p}-\frac{n+1}{2(n+p+1)}\right]}{\frac{1}{2(1+2p)}4^{-p}\binom{1+2p}{p}} \\
        &=2^{-n}\left(\prod_{j=0}^{n-2}\frac{n+2p-j}{n+p-j}\right)\frac{(n^2+n-2p)(1+2p)}{(n+2p)(n+p+1)}.
    \end{align*}
    Rewrite the product as:
    \[\prod_{j=0}^{n-2}\frac{n+2p-j}{n+p-j}=\prod_{j=0}^{n-2}\left(1+\frac{p}{n+p-j}\right).\]
    It is easily checked that for $x\in[0,n-2]$, the map $x\mapsto 1+p/(n+p-x)$ is convex. Let $\ell:[0,n-2]\to\mathbb{R}$ be the line segment between $1+p/(n+p)$ and $1+p/(2+p)$. Then $1+p/(n+p-x)\leq\ell(x)$ for $x\in[0,n-2]$; as in the proof of Proposition \ref{product_estimate}, repeatedly applying the inequality $(1+x)(1+y)\leq\big(1+(x+y)/2\big)^2$ gives:
    \[\prod_{j=0}^{n-2}\left(1+\frac{p}{n+p-j}\right)\leq\prod_{j=0}^{n-2}\ell(j)\leq\left[1+\frac{p}{2}\left(\frac{1}{n+p}+\frac{1}{2+p}\right)\right]^{n-1}.\]
    Hence, the quotient is bounded by:
    \begin{align*}
        \frac{\al{p}{n}-\al{p}{n+1}}{\al{p}{1}}&\leq 2^{-n}\left[1+\frac{p}{2}\left(\frac{1}{n+p}+\frac{1}{2+p}\right)\right]^{n-1}\frac{(n^2+n-2p)(1+2p)}{(n+2p)(n+p+1)} \\
        &\leq2^{-n}\left[1+\frac{p}{2}\left(\frac{1}{n+p}+\frac{1}{2+p}\right)\right]^{n-1}\frac{n^2+n-2p}{n+p+1} \\
        &\leq2^{-n}\left(1+\frac{n+2p}{2(n+p)}\right)^{n-1}\frac{n^2+n-2p}{n+p+1} \\
        &=2\cdot 4^{-n}\left(4-\frac{n}{n+p}\right)^{n-1}\frac{n^2+n-2p}{n+p+1} \\
        &\leq2\cdot 4^{-n}\left(4-\frac{n}{n+p}\right)^{n-1}\frac{n^2+n}{n+p}.
    \end{align*}
    We want to maximize this quantity in $p$. Define $g:[1,\infty)\times\mathbb{N}\to[0,\infty)$ by:
    \[g(x,n)=4^{-n}\left(4-\frac{n}{n+x}\right)^{n-1}\frac{n^2+n}{n+x}.\]
    Differentiating this quantity gives:
    \begin{align*}
        \frac{\partial g}{\partial x}(x,n)&=4^{-n}\left[\frac{n(n-1)}{(n+x)^2}\left(4-\frac{n}{n+x}\right)^{n-2}\frac{n^2+n}{n+x}-\left(4-\frac{n}{n+x}\right)^{n-1}\frac{n^2+n}{(n+x)^2}\right] \\
        &=4^{-n}\frac{n^2+n}{(n+x)^2}\left(4-\frac{n}{n+x}\right)^{n-2}\left(\frac{n(n-1)}{n+x}+\frac{n}{n+x}-4\right) \\
        &=4^{-n}\frac{n^2+n}{(n+x)^2}\left(4-\frac{n}{n+x}\right)^{n-2}\left(\frac{n^2}{n+x}-4\right).
    \end{align*}
    Note that all terms in the derivative are always positive except for the last term. There is precisely one point $p_0$ (not necessarily an integer) for which this partial derivative vanishes:
    \[\frac{\partial g}{\partial x}(p_0,n)=0\implies n^2-4n-4p_0=0\implies p_0=\frac{1}{4}\left(n^2-4n\right).\]
    Substituting back into our original bound on the quotient, we find that:
    \begin{align*}
        \frac{\al{p}{n}-\al{p}{n+1}}{\al{p}{1}}&\leq 2\cdot4^{-n}\left(4-\frac{n}{n+\frac{1}{4}\left(n^2-4n\right)}\right)^{n-1}\frac{n^2+n}{n+\frac{1}{4}\left(n^2-4n\right)} \\
        &=2\cdot4^{-n}\left(4-\frac{4}{n}\right)^{n-1}\left(4+\frac{4}{n}\right) \\
        &=2\left(1-\frac{1}{n}\right)^{n-1}\left(1+\frac{1}{n}\right) \\
        &=2\left(1-\frac{1}{n}\right)^n\left(1+\frac{2}{n-1}\right).
    \end{align*}
    This upper bound is decreasing in $n$, and for $n\geq 5$ is less than or equal to $1$, so we have for all $p,n\in\mathbb{N}$:
    \[\frac{\al{p}{n}-\al{p}{n+1}}{\al{p}{1}}\leq1.\]
    Hence we have for all $p,n\in\mathbb{N}$ that:
    \[\left|\al{p}{n+1}-\al{p}{n}\right|\leq\al{p}{1}.\]
    This completes the proof of the first inequality. Now to prove the second inequality, we apply the estimate from \cite{Hirschhorn} to obtain:
    \[\left|\al{p}{n+1}-\al{p}{n}\right|\leq\frac{1}{2(1+p)}4^{-p}\binom{2p}{p}\leq\frac{1}{2(1+p)}\sqrt\frac{1}{\pi p}.\]
    This completes the proof.
\end{proof}
\begin{prop}\label{uniform_convergence_prop}
    For $x\in[-1,1]$ and for $n\in\mathbb{N}$, the following power series is uniformly convergent in both $n$ and $x$:
    \[\sum_{p=0}^\infty\left(\alpha_p^{(n+1)}-\alpha_p^{(n)}\right)x^p.\]
\end{prop}
\begin{proof}
    Applying the results of Proposition \ref{uniform_p_estimate}, we have:
    \[\left|\left(\alpha_p^{(n+1)}-\alpha_p^{(n)}\right)x^p\right|\leq\frac{1}{2(1+p)}\sqrt{\frac{1}{\pi p}}\]
    for all $x\in[-1,1]$ and for all $n,p\in\mathbb{N}$. Additionally, we have:
    \[\left|\alpha_0^{(n+1)}-\alpha_0^{(n)}\right|=\left|2^{-(n+1)}-2^{-n}\right|\leq1,\]
    and hence:
    \[\left|\left(\alpha_p^{(n+1)}-\alpha_p^{(n)}\right)x^p\right|\leq\begin{cases}
    1, & p=0, \\
    \frac{1}{2(1+p)}\sqrt{\frac{1}{\pi p}}, & p\neq0. \\
    \end{cases}\]
    Define the sequence $(M_p)_{p=0}^\infty$ by:
    \[M_p=\begin{cases}
    1, & p=0, \\
    \frac{1}{2(1+p)}\sqrt{\frac{1}{\pi p}}, & p\neq0. \\
    \end{cases}\]
    Then $\left|\left(\alpha_p^{(n+1)}-\alpha_p^{(n)}\right)x^p\right|\leq M_p$ for each $n\in\mathbb{N}$ and $x\in[-1,1]$, and:
    \[\sum_{p=0}^\infty M_p\leq1+\sum_{p=1}^\infty\frac{1}{p^{3/2}}<\infty,\]
    so by the Weierstrass M-test the original sum converges uniformly in both $n$ and $x$.
\end{proof}

\section{Estimates on Coefficient Summations}
In this section we prove several technical estimates regarding summations of the coefficients $\alpha_p^{(n)}$. These estimates facilitate the proofs of the theorems in Section 5.
\begin{thm}\label{sum_coefficient_bound}
The following inequality holds:
\begin{equation}\label{main}
    \sup_{n \in \mathbb{N}} n\sum_{p=0}^\infty \left|\alpha_p^{(n+1)}-\alpha_p^{(n)}\right| \leq\frac{33}{2}.
\end{equation}
\end{thm}
\begin{proof}
Note that $K_n=0$ precisely when $n=5$, $\lfloor K_6\rfloor=0$, and $\floor{K_n}>0$ when $n>6$. Hence, in order to avoid division by zero, we consider two separate cases: when $n>6$ and when $n\leq 6$. For the first case, fix $n\in \mathbb{N}$ such that $n>6$. Expanding the left hand side of \eqref{main}, we have:

\begin{equation}\label{sum_expand}
   \sum_{p=0}^{\infty} |\alpha_p^{(n+1)} - \alpha_p^{(n)}|
= (\frac{1}{2^n}-\frac{1}{2^{n+1}})+\sum_{p=1}^{\lfloor K_n\rfloor} |\alpha_p^{(n+1)} - \alpha_p^{(n)}| + \sum_{p=\lfloor K_n\rfloor+1}^{\infty} |\alpha_p^{(n+1)} - \alpha_p^{(n)}|.
\end{equation}
By applying Lemma \ref{ijk_lemma} (4):
\begin{equation*}
    \sum_{p=1}^{\lfloor K_n\rfloor} |\alpha_p^{(n+1)} - \alpha_p^{(n)}|
      \leq K_n(\alpha_{ \lfloor K_n \rfloor}^{(n)}-\alpha_{\lfloor K_n \rfloor}^{(n+1)}). 
\end{equation*}
By Proposition \ref{uniform_p_estimate}, we have that, for $p \geq 1$: \begin{equation*}
    \left|\alpha_p^{(n+1)}-\alpha_p^{(n)}\right|\leq\frac{1}{2(1+p)}\sqrt\frac{1}{\pi p}\leq \frac{1}{p^{3/2}}.
\end{equation*}
Now, the first term of the right-hand-side of (\ref{sum_expand}) tends to zero at a rate of $2^{-(n+1)}$, faster than $C/n$ for any constant $C$. Expanding the second term, we have:
\begin{equation}\label{K_n(alpha_bound)}
      K_n (\alpha_{K_n}^{(n)}-\alpha_{K_n}^{(n+1)}) \leq \frac{K_n}{{K_n}^{3/2}} = \frac{1}{\sqrt{K_n}} \leq \frac{1}{\sqrt{\frac{n^4}{12n^2}}} = \frac{\sqrt{12}}{n}.
\end{equation}
The third term can be bounded using properties of $p$-sums. If $n > 6$, we have the following:
\begin{equation*}
        \sum_{p=\lfloor K_n\rfloor+1}^{\infty} \left| \alpha_p^{(n)}-\alpha_p^{(n+1)} \right| \leq \sum_{p=\lfloor K_n\rfloor+1}^{\infty}\frac{1}{p^{3/2}}.
\end{equation*}
Using the integral test, we can bound this estimate as follows:
\begin{equation*}
        \sum_{p=\lfloor K_n\rfloor+1}^{\infty}\frac{1}{p^{3/2}} \leq \int_{\lfloor K_n\rfloor}^{\infty} \frac{1}{x^{3/2}}dx = \frac{2}{\sqrt{\lfloor K_n\rfloor}} \leq  \frac{2\sqrt{12}}{n},
\end{equation*}
and combining the above estimates yields:
\begin{equation*}
    \sum_{p=0}^\infty |\alpha_p^{(n)}-\alpha_p^{(n+1)}| \leq (\frac{1}{2^n}-\frac{1}{2^{n+1}}) + (\frac{\sqrt{12}}{n} + \frac{2\sqrt{12}}{n}) \leq  \frac{7\sqrt{3}}{n}.
\end{equation*}
Now, consider: \begin{equation*}
     D_n =\sum_{p=0}^\infty |\alpha_p^{(n)} - \alpha_p^{(n+1)}| \leq \frac{3}{4} + \sum_{p=2}^\infty |\alpha_p^{(n)} - \alpha_p^{(n+1)}|
 \end{equation*}
 for all $1\leq n \leq 6$. Using the integral test again, we get: \begin{equation*}
     D_n \leq \frac{3}{4} + \int_1^\infty \frac{1}{x^{3/2}}dx = \frac{11}{4}.
 \end{equation*}
Clearly, then: \begin{equation*}
    D_n \leq \frac{6\cdot11}{4n}.
\end{equation*}
Letting $C = 33/2$, we are done. 

\end{proof}

%%END ESTIMATE%%

\begin{cor}\label{uniform_bound_corollary}
For all $n\geq 1$ and all x $\in [0,1]$, $n|\psi(x)^{n+1}-\psi(x)^n|$ is uniformly bounded. 
\end{cor}
\begin{proof}
Since:
\begin{equation*}
    |\psi(x)^{n+1}-\psi(x)^n|\leq \sum_{p=0}^\infty |\alpha_p^{(n+1)}-\alpha_p^{(n)}||x^p| \leq \sum_{p=0}^\infty |\alpha_p^{(n+1)}-\alpha_p^{(n)}| \leq \frac{33}{2n},
\end{equation*} the result follows. 
\end{proof}

See Figure 2 for a graphical representation of $|\psi(x)^n-\psi(x)^{n+1}|$.

\begin{figure}[htp]
    \centering
    \includegraphics[width=10cm]{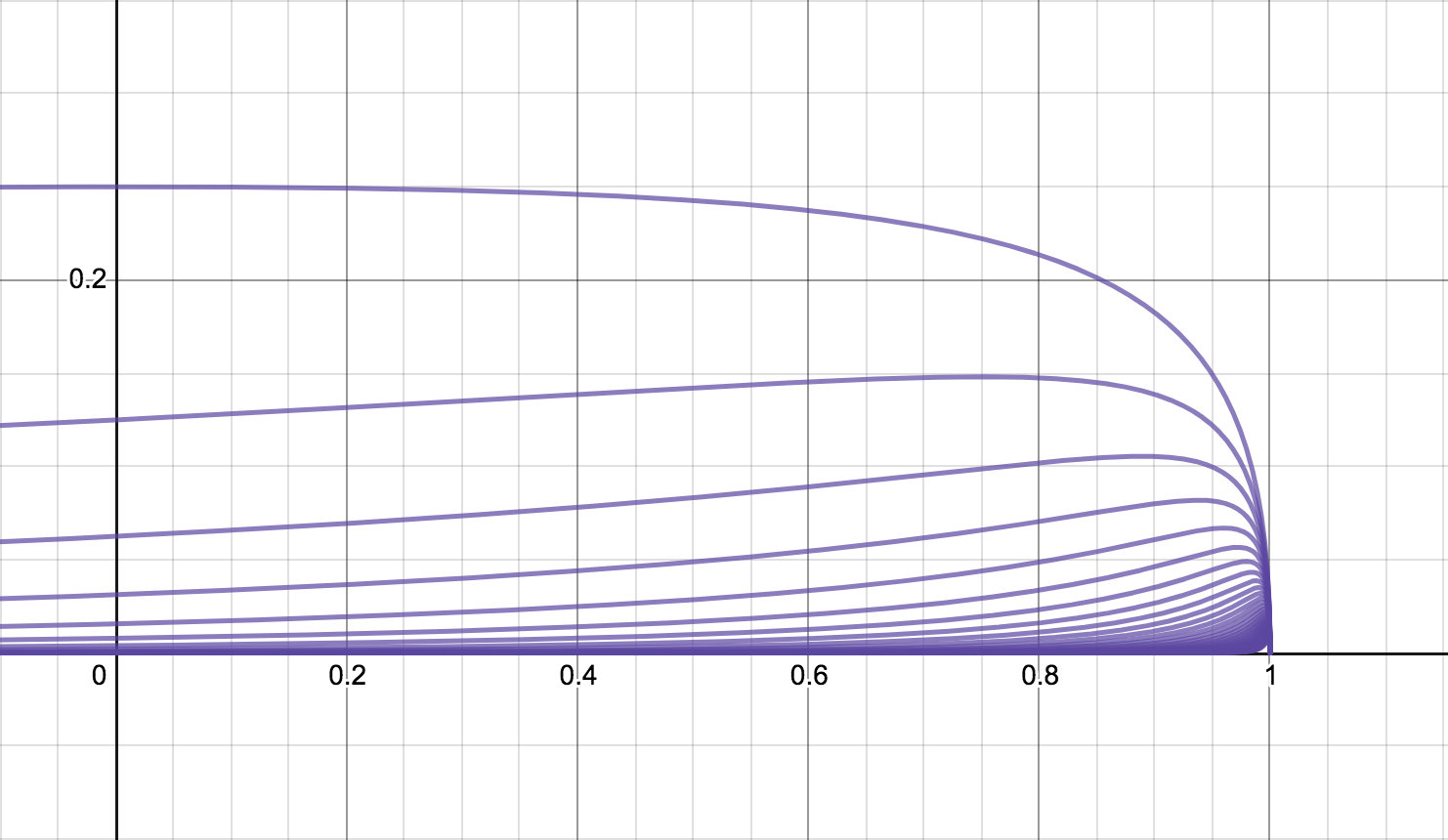}
    \caption{Graph of ($x$, $|\psi(x)^n-\psi(x)^{n+1}| $) for $n$ = 1, 2, ..., 40. Graphed with desmos. }
    \label{fig:Thm1}
\end{figure}

\begin{cor}\label{uniform_bound_corollary_2}
$|(n+1)\psi(x)^{n+1} -n\psi(x)^{n}| < 35/2 $ for all $n\geq{1}$ and for all $x\in[-1,1]$.
\end{cor}

\begin{proof}
By definition:
 \begin{equation*}
        |(n+1)\psi(x)^{n+1} -n\psi(x)^{n}| \leq \bigg{|}\sum_{p=0}^\infty n(\alpha_p^{(n+1)} -\alpha_p^{(n)}) + \alpha_p^{n+1}\bigg{|}
        \leq\frac{33}{2}+1=\frac{35}{2},
 \end{equation*}
so the result follows.
\end{proof}

Figure 3 here gives a visualization of why the difference does not go to zero. We see an issue as we approach $x =1$. 

\begin{figure}[htp]
    \centering
    \includegraphics[width=10cm]{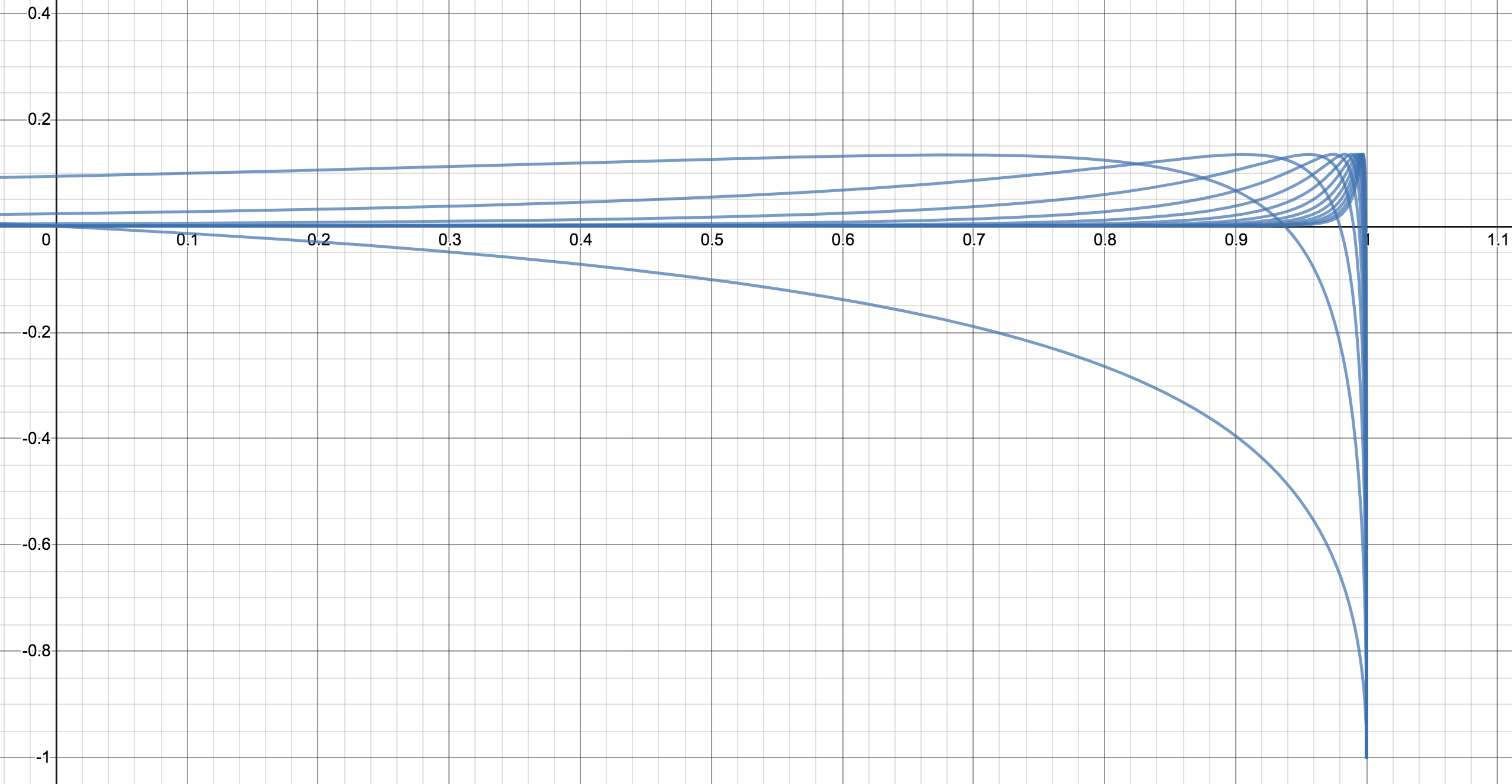}
    \caption{Graph of ($x$, $n\psi(x)^n-(n+1)\psi(x)^{n+1} $) for $n$ = 1, 2, ..., 40. Note that as $x$ $\to$ 1, the function does not approach 0. Graphed with desmos. }
    \label{fig:Var1}
\end{figure}

We now state a more precise version of $K_n$ as defined in Lemma \ref{ijk_lemma} (3).
\begin{lem}\label{Upper_Lower_Kn}
    Consider the zeros of the following function $f_n$:
    \[f_n(p) = p(12n^2 + 12n -12)-12p^2 - (n^4 -2n^3-13n^2 -10n).\]
    Denote $\underline{K_n}$ and $\overline{K_n}$ as the smaller and larger roots, respectively. Then, for all $n \in \mathbb{N}$, we have \begin{enumerate}
        \item $\alpha_p^{(n)} - \alpha_p^{(n+1)} \leq \alpha_{p+1}^{(n)}-\alpha_{p+1}^{(n+1)}$ if and only if $p \leq \underline{K_n}$ or $p \geq \overline{K_n}$.
        \item $\alpha_p^{(n)} - \alpha_p^{(n+1)} \geq \alpha_{p+1}^{(n)}-\alpha_{p+1}^{(n+1)}$ if and only if $\underline{K_n} \leq p \leq \overline{K_n}$.
    \end{enumerate}
\end{lem}
\begin{proof}
    See the appendix for details.
\end{proof}
The following proposition gives an upper bound on the growth rate of the sum $\sum_{p=0}^\infty\left|\alpha_{p+1}^{(n)}-\alpha_p^{(n)}\right|(p+1)$ as $n\to\infty$ and is crucial to the proof of Theorem \ref{mean_bound_bound}. This theorem was proved in \cite{Brunel} without an explicit constant in the upper bound; we provide a full proof of this theorem (with the explicit constant) in the appendix.
\begin{thm}\label{sup_2_bound}
The following inequality holds:
\begin{equation*}
   \sup_{n \in \mathbb{N}} \sum_{p=0}^\infty |\alpha_p^{(n)} - \alpha_{p+1}^{(n)}|(p+1) \leq 3 + \frac{5}{\sqrt{\pi}}. 
\end{equation*}
\end{thm}

\begin{thm}\label{sup_3_bound}
Let $\zeta_p^{(n)} = \alpha_{p}^{(n)} - \alpha_{p}^{(n+1)}$. Then \begin{equation}
    \sup_{n \in \mathbb{N}} n \sum_{p = 0}^\infty |\zeta_p^{(n)} - \zeta_{p+1}^{(n)}| (p+1)\leq33+4\sqrt3.
\end{equation}
\end{thm}
\begin{proof}
 Fix $n, N \in \mathbb{N}$, $n > 6$ such that $N > \overline{K_n}$. In the proof we will use known properties of $K_n$, and splitting into the cases when $n>6$ and $n\leq 6$ will help us avoid problems with dividing by $0$.  Set $\Delta_p^{(n)} = \zeta_{p+1}^{(n)}-\zeta_{p}^{(n)}$. Summing by parts, we have:

\begin{align*}
    \sum_{p=0}^N |\Delta_p^{(n)}|(p+1) & =  \underbrace{(N+1) \sum_{p=0}^N |\Delta_p^{(n)}|}_{(\dagger)}- \underbrace{\sum_{p =0}^{N-1}\sum_{\ell = 0}^p |\Delta_\ell^{(n)}|}_{(\dagger\dagger)}. 
\end{align*}

For the first sum, we have: \begin{align*}
    (\dagger) &= (N+1)\left(\sum_{p=0}^{\lfloor\underline{K_n}\rfloor}\Delta_p^{(n)} - \sum_{p=\lfloor\underline{K_n}\rfloor+1}^{\lfloor \overline{K_n}\rfloor}\Delta_p^{(n)} + \sum_{p=\lfloor \overline{K_n}\rfloor+1}^{N} \Delta_p^{(n)}\right)\\
    &= (N+1)\left(2\zeta^{(n)}_{\lfloor\underline{K_n}\rfloor+1}+ \zeta^{(n)}_{N+1} - 2\zeta^{(n)}_{\lfloor\overline{K_n}\rfloor+1}-\zeta^{(n)}_{0}\right)
\end{align*}

For the second sum, we have: \begin{align*}
     (\dagger\dagger) & =\sum_{p=0}^{\lfloor\underline{K_n}\rfloor}\sum_{\ell = 0}^p \Delta_{\ell}^{(n)} + \sum_{p=\lfloor\underline{K_n}\rfloor+1}^{\lfloor \overline{K_n}\rfloor}\left(\sum_{\ell=0}^{\lfloor \underline{K_n}\rfloor} \Delta_\ell^{(n)} - \sum_{\ell = \lfloor \underline{K_n}\rfloor+1}^p\Delta_\ell^{(n)}\right) \\
     &\hspace{0.5in }+ \sum_{p=\lfloor \overline{K_n}\rfloor+1}^{N} \left(\sum_{\ell=0}^{\lfloor \underline{K_n}\rfloor} \Delta_\ell^{(n)} - \sum_{\ell = \lfloor \underline{K_n}\rfloor+1}^{\lfloor \overline{K_n} \rfloor}\Delta_\ell^{(n)} + \sum_{\ell = \lfloor \overline{K_n}\rfloor + 1}^{p} \Delta_\ell^{(n)}\right)\\
     &=\sum_{p=0}^{\lfloor\underline{K_n}\rfloor}\left(\zeta^{(n)}_{p+1} -\zeta_{0}^{(n)}\right)
     +\sum_{p=\lfloor\underline{K_n}\rfloor+1}^{\lfloor \overline{K_n}\rfloor}\left(2\zeta_{\lfloor\underline{K_n} \rfloor+1}^{(n)} - \zeta_0^{(n)} -\zeta_{p+1}^{(n)}\right)\\
     &\hspace{0.5in} +  \sum_{p=\lfloor \overline{K_n}\rfloor+1}^{N-1} \left(2\zeta_{\lfloor\underline{K_n} \rfloor+1}^{(n)}-2\zeta_{\lfloor\overline{K_n} \rfloor+1}^{(n)} - \zeta_0^{(n)} +\zeta_{p+1}^{(n)}\right)\\
     &= \left(\sum_{p=0}^{N-1} \zeta^{(n)}_{p+1} - 2\sum_{p=\lfloor\underline{K_n}\rfloor+1}^{\lfloor \overline{K_n}\rfloor}\zeta^{(n)}_{p+1}\right)- N\zeta_0^{(n)} + (N -\lfloor \underline{K_n}\rfloor -1)\left(2\zeta_{\lfloor\underline{K_n} \rfloor+1}^{(n)}\right)\\
     & \hspace{0.5in} -\left(N -\lfloor \overline{K_n}\rfloor -1\right)\left(2\zeta_{\lfloor\overline{K_n} \rfloor+1}^{(n)}\right).\\
     &= \left(\sum_{p=0}^{N} \zeta^{(n)}_{p} - 2\sum_{p=\lfloor\underline{K_n}\rfloor+1}^{\lfloor \overline{K_n}\rfloor}\zeta^{(n)}_{p}\right)- (N+1)\zeta_0^{(n)} \\
     &\hspace{0.5in}+((N+1) -\lfloor \underline{K_n}\rfloor -1)\left(2\zeta_{\lfloor\underline{K_n} \rfloor+1}^{(n)}\right)\\
     & \hspace{0.5in} -\left((N+1) -\lfloor \overline{K_n}\rfloor -1\right)\left(2\zeta_{\lfloor\overline{K_n} \rfloor+1}^{(n)}\right).\\
\end{align*}
Combining both terms back yields: 
\begin{align*}
    \sum_{p=0}^N |\Delta_p^{(n)}|(p+1) &= (N+1) \zeta_{N+1}^{(n)}+\left( 2\sum_{p=\lfloor\underline{K_n}\rfloor+1}^{\lfloor \overline{K_n}\rfloor}\zeta^{(n)}_{p}- \sum_{p=0}^{N-1} \zeta_{p}^{(n)}\right) \\
    &+\left(\lfloor \underline{K_n}\rfloor +1\right)\left(2\zeta_{\lfloor\underline{K_n} \rfloor+1}^{(n)}\right) - \left(\lfloor \overline{K_n}\rfloor +1\right)\left(2\zeta_{\lfloor\overline{K_n}\rfloor + 1}^{(n)}\right).
\end{align*}
Letting $N \to \infty$, we have: \begin{align*}
    \sum_{p=0}^\infty |\Delta_p^{(n)}|(p+1) &= \left( 2\sum_{p=\lfloor\underline{K_n}\rfloor+1}^{\lfloor \overline{K_n}\rfloor}\zeta^{(n)}_{p}\right)+2\left(\lfloor \underline{K_n}\rfloor +1\right)\left(\zeta_{\lfloor\underline{K_n} \rfloor+1}^{(n)}\right) \\
    &\hspace{0.5in}- 2\left(\lfloor \overline{K_n}\rfloor +1\right)\left(\zeta_{\lfloor\overline{K_n}\rfloor + 1}^{(n)}\right)\\
    & \leq 2\sum_{p=\lfloor\underline{K_n}\rfloor+1}^{\lfloor \overline{K_n}\rfloor}\zeta^{(n)}_{p}+2\left(\lfloor \underline{K_n}\rfloor +1\right)\zeta_{\floor{K_n}+1}^{(n)}\\
    &\leq  \frac{33}{n} + \frac{2}{\sqrt{\lfloor \underline{K_n}\rfloor+1}}\\
    &\leq \frac{33}{n} + \frac{2}{\sqrt{K_n}}\\
    &\leq \frac{33 + 4\sqrt{3}}{n}
\end{align*}
When $n \leq 6$, $\underline{K_n} < 0$, and we obtain some reductions to our bound. In particular \begin{align*}
    \sum_{p=0}^\infty |\Delta_p^{(n)}|(p+1) &\leq \frac{33}{n} + 2\left(\frac{1}{\sqrt{\lfloor \overline{K_n}\rfloor+1}}\right)\\
    &\leq \frac{33}{n} + 2\\
    & = \frac{33 + 2n}{n}\\
    &\leq \frac{33 + 4\sqrt{3}}{n}.
\end{align*}
Hence $\sup_{n \in \mathbb{N}}n\sum_{p=0}^\infty |\Delta_p^{(n)}|(p+1) \leq 33 + 4\sqrt{3}< \infty$.
\end{proof}
The following result gives a closed form for the partial sums on $n$ of the coefficients $\alpha_p^{(n)}$.
\begin{prop}\label{sum_closed_form}
    For any $N\in\mathbb{N}^+$ and $p\in\mathbb{N}$, we have:
    \begin{equation}\label{eq:sum_powers}
        \sum_{n=1}^N\alpha_p^{(n)}=2^{-(N+2p)}\left[2^N\binom{2p}{p}-\binom{N+2p}{p}\right].
    \end{equation}
\end{prop}
\begin{proof}
    Define:
    \[\gamma_p^{(n)}=2^{-(n+2p)}\binom{n+2p}{p}.\]
    Then we have:
    \begin{align*}
        \gamma_p^{(n-1)}-\gamma_p^{(n)}&=2^{-(n+2p)}\left[2\binom{n+2p-1}{p}-\binom{n+2p}{p}\right] \\
        &=2^{-(n+2p)}\left[2\binom{n+2p}{p}\frac{n+p}{n+2p}-\binom{n+2p}{p}\right] \\
        &=2^{-(n+2p)}\binom{n+2p}{p}\left[\frac{2(n+p)}{n+2p}-1\right] \\
        &=\frac{n}{n+2p}2^{-(n+2p)}\binom{n+2p}{p} \\
        &=\alpha_p^{(n)}.
    \end{align*}
    Hence, the series can be written as:
    \[\sum_{n=1}^N\alpha_p^{(n)}=\sum_{n=1}^N\left(\gamma_p^{(n-1)}-\gamma_p^{(n)}\right)=\gamma_p^{(0)}-\gamma_p^{(N)}=2^{-(N+2p)}\left[2^N\binom{2p}{p}-\binom{N+2p}{p}\right],\]
    which yields the desired formula.
\end{proof}

\section{Properties of the Brunel Operator}
In this section we prove several fundamental estimates for the Brunel operator. We begin by discussing an important recurrence property of the Brunel operator, and apply this recurrence to show that $A(T)$ is invertible for any operator $T$ for which $A(T)$ exists. Note that below we define $A^0(T)=I$.
\begin{prop}\label{brunel_operator_recurrence}
For all $n\geq2$, we have:
\begin{equation*}
    A^n(T)T = 2A^{n-1}(T)-A^{n-2}(T).
\end{equation*}
\end{prop}
\begin{proof}
For $n \geq 2$, we have by Lemma \ref{alpha_reurrence} that $\alpha_p^{(n)} = 2\alpha_{p+1}^{(n-1)} - \alpha^{(n-2)}_{p+1}$. Therefore, we have:
\begin{align*}
    A^n(T)T &= \sum_{p=0}^\infty \alpha_p^{(n)} T^{p+1}\\
    &= \sum_{p=0}^\infty \left(2\alpha_{p+1}^{(n-1)} - \alpha_{p+1}^{(n-2)}\right) T^{p+1} \\
    &= \sum_{p=1}^\infty \left(2 \alpha_p^{(n-1)}-\alpha_p^{(n-2)}\right)T^p. 
\end{align*}
But $2\alpha_{0}^{(n-1)} - \alpha_{0}^{(n-2)} = 0$, so in fact: \begin{equation*}
    A^n(T)T = \sum_{p=0}^\infty \left(2 \alpha_p^{(n-1)}-\alpha_p^{(n-2)}\right)T^p = 2A^{n-1}(T)-A^{n-2}(T),
\end{equation*}
as claimed.
\end{proof}
\begin{prop}\label{inversion_formula}
    Let $T:X\to X$. Then the Brunel operator $A(T):X\to X$ is invertible whenever it exists, and we have $A^{-1}(T)=2I-A(T)T$.
\end{prop}
\begin{proof}
    Note that by the recurrence in Proposition \ref{brunel_operator_recurrence}, we have:
    \[A(T)\big(2I-A(T)T\big)=I,\]
    so it follows that $A^{-1}(T)=2I-A(T)T$.
\end{proof}
In \cite{Lootgieter}, Lootgieter proved that for $T:L^p(X,\mu)\to L^p(X,\mu)$, $A$ is a Ritt operator for the case of $T$ mean-bounded and positive. In the sequel we drop the condition of positivity and show that for $X$ a Banach space, $A$ is a Ritt operator for all $T \in \mathscr{L}(X)$ mean-bounded. While our main focus is on mean-bounded operators $T$, we first show that $A$ is a Ritt operator for $T$ power-bounded as an important special case in the following theorem. We note here that this theorem does follow from Theorem 4.1 of \cite{Dungey}, but we give an alternate proof here to illustrate how the study of the Brunel operator reduces to the study of the coefficients $\al{p}{n}$; our proof uses elementary techniques and requires only the properties of the coefficients discussed in the previous sections.
\begin{thm}\label{power_bound_bound}
Let $T$ be an operator on a Banach space $X$. Let $A=\psi(T)$ be the Brunel operator defined by $T$. If $T$ is power-bounded, then $A$ is power-bounded, and
    \begin{equation*}
    \sup_{n\in\mathbb{N}}n\|A^n-A^{n+1}\|\leq C.
\end{equation*}
If $\left\Vert T^n\right\Vert\leq M$ for all $n\in\mathbb{N}$, then $C$ can be taken to be $33M/2$.
\end{thm}
\begin{proof}
Since $T$ is power-bounded, there exists an $M < \infty$ such that $\|T^p\| \leq M$ for all $p \in \mathbb{N}$. Fix $n \in \mathbb{N}$; then \begin{equation}
    \|A^n\| \leq \sum_{p=0}^\infty \alpha_p^{(n)}\left\Vert T^p\right\Vert \leq M\sum_{p=0}^\infty\alpha_p^{(n)}=M.
\end{equation}
Thus $A$ is power-bounded. Now let $f \in X$ where $\|f\| \leq 1$. We then have: \begin{equation*}
    n\|A^nf-A^{n+1}f\| = n\left\|\sum_{p=0}^\infty (\alpha_p^{(n)}-\alpha_p^{(n+1)}) T^pf\right\| \leq n\sum_{p=0}^\infty \left|\alpha_p^{(n)}-\alpha_p^{(n+1)}\right| \|T^pf\|.
\end{equation*}
Again, since $T$ is power-bounded, we have that $\left\|T^pf\right\| \leq M < \infty$ for all $p\in \mathbb{N}$ and $f$ such that $\|f\| \leq 1$. Thus:
\begin{equation*}
    n\sum_{p=0}^\infty \left|\alpha_p^{(n)}-\alpha_p^{(n+1)}\right| \|T^pf\| \leq Mn\sum_{p=0}^\infty \left|\alpha_p^{(n)}-\alpha_p^{(n+1)}\right| \leq \frac{33M}{2} < \infty.
\end{equation*}
As this bound holds for all $n\in\mathbb{N}$, we see that:
\begin{equation*}
    \sup_{n\in\mathbb{N}}n\|A^nf-A^{n+1}f\|\leq\sup_{n\in\mathbb{N}}n\sum_{p=0}^\infty \left|\alpha_p^{(n)}-\alpha_p^{(n+1)}\right| \|T^pf\|\leq \frac{33M}{2} < \infty.
\end{equation*}
Since $f$ was arbitrary, the theorem holds. 
\end{proof}
We now wish to extend this theorem to the case where $T$ is mean-bounded. 

\begin{thm}\label{mean_bound_bound}
    Let $X$ be a Banach space and let $T:X\to X$ be mean-bounded (not necessarily positive). Then $A = \psi(T)$ is power-bounded, and there exists a constant $C<\infty$ such that:
    \begin{equation}{\label{difference_bound}}
        \sup_{n\in\mathbb{N}}n\left\Vert A^{n+1}-A^n\right\Vert\leq C.
    \end{equation}
    If $\left\Vert M_n(T)\right\Vert\leq M$ for all $n\in\mathbb{N}$, the constant $C$ can be taken to be $M(33+4\sqrt3)$.
\end{thm}
\begin{remark}
    Note that Brunel and Emilion provided a sketch of a proof that $A$ is power-bounded which explicitly assumes that $T$ is positive (see page 104 of \cite{Brunel}); a proof for 2.3 on page 105 is not provided. Brunel and Emilion estimate the norm of the partial sums for the Brunel operator $A^n(T)$, which suffices to prove the result for the case when $T$ is positive, but in the general case a more subtle argument is required to show that the limit of the partial sums exists. The proof below instead shows that the partial sums are Cauchy in norm, which is now a straightforward application of our estimates on the coefficients $\al{p}{n}$.
\end{remark}
\begin{proof}
To show that $A(T)$ is power-bounded, we use an Abel summation and Theorem \ref{sup_2_bound} to show that the norms of the partial sums of $A^n(T)$ form a Cauchy sequence. To prove Equation \eqref{difference_bound}, we utilize Theorem \ref{sum_coefficient_bound} and Lemma \ref{Upper_Lower_Kn}. Recall that: \begin{equation}
    M_N = M_N(T) = \frac{I + \cdots + T^{N-1}}{N}.
\end{equation}
Set $M_0 = 0$. An Abel summation yields \begin{align*}
    \sum_{p =0}^{N}\alpha_p^{(n)}T^p = \sum_{p = 0}^{N-1}\left(\alpha_{p}^{(n)}- \alpha_{p+1}^{(n)}\right)(p+1)M_{p+1} + \alpha_{N}^{(n)}(N+1)M_{N+1}.
\end{align*}
Let $M =\sup_n\|M_n(T)\|$, fix $f \in X$, and set \[\left(S_N^{(n)}f\right)_{N\in\mathbb{N}}\equiv \left(\sum_{p=0}^N \alpha_p^{(n)}T^pf\right)_{N\in\mathbb{N}}.\] For $L \in \mathbb{N}$ such that $N \leq L$, we have: \begin{align*}
    \left\|S_L^{(n)}f - S_N^{(n)}f\right\|&\leq M \Bigg[\sum_{p = N}^{L-1}{\left|\alpha_{p}^{(n)}-\alpha_{p+1}^{(n)}\right|(p+1)} \\
    &\hspace{0.5in}+ \left|(L+1) \alpha_{L}^{(n)} - (N+1)\alpha_N^{(n)}\right| \Bigg]\|f\|.
\end{align*} 
Since $N\alpha_N^{(n)} \to 0$  as $N \to \infty$ (Corollary \ref{N_alpha_N_estimate}), Theorem \ref{sup_2_bound} tells us that $S_N^{(n)}f$ is a Cauchy sequence in $X$. 
Hence, $\left(S_N^{(n)}f\right)_{N\in\mathbb{N}}$ is a Cauchy sequence for every $n\in \mathbb{N}$ and every $f\in X$. Since $X$ is a Banach space, $\left(S_N^{(n)}f\right)_{N\in\mathbb{N}}$ converges to $A^n f$ in $X$. We also have that, for all $n \in \mathbb{N}$: \begin{equation}
    \|S_N^{(n)}\| \leq M \sum_{p =0}^{N-1}|\alpha_{p}^{(n)}-\alpha_{p+1}^{(n)}|(p+1) + (N+1)\alpha_{N}^{(n)} < \infty.
\end{equation}
   
Fix $\varepsilon >0$, $f \in X$ such that $\|f\| \leq 1$, and $n \in \mathbb{N}$. 
Choose $N$ large enough so that $\|A^n f-S_N^{(n)}f\|<\varepsilon/(M+1)$ and $(N+1)\alpha_N^{(n)}<\varepsilon/(M+1)$. Then, we have: \begin{align*}
    \left\|A^{n}f\right\| &\leq \left\|A^{n}f - S_N^{(n)}f\right\| + \left\|S^{(n)}_{N}f\right\|\\
    & \leq \left\|A^{n}f - S_N^{(n)}f\right\|+ M\left[1 + \frac{5}{\sqrt{\pi}}+ (N+1)\alpha_N^{(n)} \right]\\
    & \leq\frac{\varepsilon}{M+1} + M\left[1 +\frac{5}{\sqrt{\pi}}+ \frac{\varepsilon}{M+1}\right]\\
    & \leq M\left[1 + \frac{5}{\sqrt{\pi}}\right]+\varepsilon.
\end{align*}
Since $\varepsilon$ was arbitrary, we indeed have $\|A^nf\| \leq  M\left[1 + \frac{5}{\sqrt{\pi}}\right]$ for all $f$ with $\|f\|\leq 1$, and thus $A$ is power-bounded.

Set $\zeta_{p}^{(n)} = \alpha_{p}^{(n)}-\alpha_{p}^{(n+1)}$. For \eqref{difference_bound}, another Abel summation
yields: 
\begin{align*}
    \sum_{p=0}^N\zeta_{p}^{(n)}T^p &= \sum_{p = 0}^{N-1}(\zeta_p^{(n)}-\zeta_{p+1}^{(n)})(p+1)M_{p+1} + \zeta_{N}^{(n)}(N+1)M_{N+1}.
\end{align*}
    Thus, \begin{align*}
        \left\|\sum_{p=0}^N\zeta_{p}^{(n)}T^p\right\| &\leq M\left(\sum_{p=0}^{N-1}\left|\zeta_{p}^{(n)} - \zeta_{p+1}^{(n)}\right|(p+1) + (N+1)\left|\zeta_{N}^{(n)}\right|\right).
    \end{align*}
    Note, for all $f \in X$, \[\left(S_N^{(n)}f - S_{N}^{(n+1)}f\right)\to_N A^{(n)}f - A^{(n+1)}f.\] Now let $f\in X$ with $\Vert f\Vert\leq1$, and fix $\varepsilon>0$ and $n\in\mathbb{N}$. Choose $N$ large enough such that $\Vert(A^{n+1}f-A^n f)-(S_N^{(n+1)}f-S_N^{(n)}f)\Vert<\varepsilon/(M +1)$ and $(N+1)\left|\zeta_N^{(n)}\right|< \varepsilon/(M+1)$. By the triangle inequality, we have by Theorem \ref{sup_3_bound}:
    \begin{align*}
        \Vert A^{n+1}f-A^nf\Vert&\leq\Vert(A^{n+1}f-A^n f)-(S_N^{(n+1)}f-S_N^{(n)}f)\Vert+\Vert S_N^{(n+1)}-S_N^{(n)}\Vert \\
        &<\frac{\varepsilon}{M+1}+\left\|\sum_{p=0}^N\zeta_{p}^{(n)}T^p\right\|\\
        &\leq \frac{\varepsilon}{M+1}+  M\left(\frac{33 + 4\sqrt{3}}{n} + \frac{\varepsilon}{M+1}\right)\\
        & \leq \varepsilon + M\left(\frac{33 + 4\sqrt{3}}{n}\right).
    \end{align*}
    As $\varepsilon$ was arbitrary, it follows that, for all $f \in X$, $\|f\|\leq 1$:
    \[\Vert A^{n+1}f-A^nf\Vert\leq M\left(\frac{33 + 4\sqrt{3}}{n}\right),\]
    and therefore $\sup_{n \in \mathbb{N}}n\Vert A^{n+1}-A^n\Vert\leq M\left(33 + 4\sqrt{3}\right)< \infty$.
\end{proof}
\begin{cor}\label{pointwise_convergence}
    Let $(X,\mathcal{M},\mu)$ be a measure space, and for some $p\in(1,\infty]$, let $T:L^p(X,\mu)\to L^p(X,\mu)$ be mean-bounded (but not necessarily positive). Then for $f\in L^p(X,\mu)$, $A^n(T)f(x)-A^{n+1}(T)f(x)\to0$ for $\mu$-almost every $x\in X$.
\end{cor}
\begin{proof}
    Let $p\in(1,\infty)$. By Equation \eqref{difference_bound}, we have:
    \[\int_X\left|A^n(T)f-A^{n+1}(T)f\right|^p\,d\mu\leq\frac{C}{n^p}.\]
    Hence, the following sum is finite:
    \[\sum_{n=1}^\infty\int_X\left|A^n(T)f-A^{n+1}(T)f\right|^p\,d\mu<\infty.\]
    Interchanging the sum and integral by the monotone convergence theorem, we find that:
    \[\int_X\sum_{n=1}^\infty\left|A^n(T)f-A^{n+1}(T)f\right|^p\,d\mu<\infty.\]
    This implies that:
    \[\sum_{n=1}^\infty\left|A^n(T)f-A^{n+1}(T)f\right|^p<\infty\quad\mu\text{-almost everywhere}.\]
    But this sum can converge only if the tail of the series converges to zero, which implies that for $\mu$-almost every $x\in X$:
    \[\lim_{n\to\infty}\left|A^n(T)f(x)-A^{n+1}(T)f(x)\right|^p=0,\]
    and so the result follows. For the case $p=\infty$, note that:
    \[\left|A^n(T)f(x)-A^{n+1}(T)f(x)\right|\leq\Vert A^{n+1}(T)-A^n(T)\Vert\cdot\Vert f\Vert_{L^\infty}\]
    for $\mu$-almost every $x\in X$, and so the desired conclusion follows by letting $n\to\infty$ and using Equation \eqref{difference_bound}.
\end{proof}
\begin{cor}\label{ritt}
    For $X$ a (complex) Banach space and $T:X\to X$ mean-bounded, the Brunel operator $A=\psi(T)$ is a Ritt operator.
\end{cor}
\begin{remark}
    Theorem 1.3 (III) of \cite{Dungey} implies that for any mean-bounded $T:X\to X$, there exists a power-bounded operator $S:X\to X$ and an $\alpha\in(0,1)$ such that $\psi(T)=I-(I-S)^\alpha$. See \cite{Dungey} and  \cite{Tomilov} for an in-depth treatment of Ritt operators and their relationship to power-bounded operators.
\end{remark}
\begin{cor}\label{ritt_cor}
    Let $X$ be a Banach space, $T:X\to X$ mean-bounded. Then the following hold:
    \begin{enumerate}
        \item $\sup_{n\in\mathbb{N}}n^2\left\Vert A^n(T)\big(I-A(T)\big)^2\right\Vert<\infty$;
        \item $\sup_{n\in\mathbb{N}}n^2\left\Vert A^n(T)(I-T)\right\Vert<\infty$.
    \end{enumerate}
    Further, if $X=L^p(Y,\mu)$, $1\leq p\leq\infty$, then both $A^n(T)\big(I-A(T)\big)^2f$ and $A^n(T)(I-T)f$ converge $\mu$-almost everywhere to zero as $n\to\infty$ for any $f\in L^p(Y,\mu)$.
\end{cor}
\noindent Note that, unlike in Corollary \ref{pointwise_convergence}, the proof below shows that the stated convergence also holds when $p=1$.
\begin{proof}
    For (1), we have the estimate:
    \begin{align*}
        n^2\left\Vert A^{2n}(T)\big(I-A(T)\big)^2\right\Vert&=\left\Vert \left(nA^n(T)\big(I-A(T)\big)\right)\left(nA^n(T)\big(I-A(T)\big)\right)\right\Vert \\
        &\leq n\left\Vert A^n(T)\big(I-A(T)\big)\right\Vert\cdot n\left\Vert A^n(T)\big(I-A(T)\big)\right\Vert \\
        &\leq C^2
    \end{align*}
    by Theorem \ref{mean_bound_bound}. This proves the bound for $n$ even. For $n$ odd, write $n=2m+1$; we have:
    \begin{align*}
        n^2\left\|A^n(T)\big(I-A(T)\big)^2\right\|&=(2p+1)^2\left\|A^{2p+1}(T)\big(I-A(T)\big)^2\right\| \\
        &\leq5\|A\|\cdot p^2\left\|A^{2m}(T)\big(I-A(T)\big)^2\right\| \\
        &\leq5\|A\|\cdot C^2,
    \end{align*}
    so the bound holds for $n$ odd as well.
    
    For (2), note that $(I-A)^2=I-2A+A^2=-TA^2+A^2=A^2(I-T)$ by the recurrence in Proposition \ref{brunel_operator_recurrence}. Hence, we have:
    \[\sup_{n\in\mathbb{N}}n^2\left\|A^{n+2}(T)(I-T)\right\|\leq\sup_{n\in\mathbb{N}}n^2\left\|A^n(T)\big(I-A(T)\big)^2\right\|<\infty,\]
    so the estimate holds. The same argument used in the proof of Corollary \ref{pointwise_convergence} shows that if $X=L^p(Y,\mu)$ and $1\leq p\leq\infty$, then for any $f\in L^p(Y,\mu)$, both $A^n(T)\big(I-A(T)\big)^2f$ and $A^n(T)(I-T)f$ converge $\mu$-almost everywhere to zero as $n\to\infty$.
\end{proof}

Consider the following operator $T$ in $\mathbb{R}^2$:
\[T=\begin{pmatrix}
- 1 & 2 \\
0 & -1 \\
\end{pmatrix}.\]
Notice that $T$ is not power-bounded, but it is mean-bounded. However, there are $x \in \mathbb{R}^2$ such that $M_N(T)x$ doesn't even converge. In spite of this, as the above theorem shows, we have $\sup_n\|A^n\|< \infty$. A direct computation of the Brunel operator for this example is done in the appendix.

We now have an important result that relates the Cesàro averages of $T$ to the Cesàro averages of $A\equiv\psi(T)$. When $T$ is mean-bounded and positive, the result reduces the study of the maximal function for the averages of $T$ to the maximal function for the averages of $A$. By doing so, it reduces the study of the norm or pointwise convergence of mean-bounded positive operators to the study of the pointwise convergence of positive power-bounded operators. More on this relationship is given in \cite{preprint}.

The following theorem was stated in \cite{Brunel} without details. In \cite{Lootgieter}, Lootgieter provides an outline of the proof using Stirling's formula ((12), pg. 4301) and the asymptotic equivalence given in \cite{Krengel} (Lemma 3.2, pg. 212). We provide a more computationally driven proof using Proposition \ref{sum_closed_form} which allows us to determine an upper bound with an explicit constant. 

\begin{remark}
    Note that in the following, $[x] = \lfloor x + 1\rfloor$, i.e  $[x]$ is the integer part of $x + 1$. 
\end{remark}
\begin{thm}\label{cesaro_square_root_average_bound}
    Let $T$ be a positive mean-bounded operator on a Banach lattice $X$. Then for any $n\in\mathbb{N}^+$:
    \begin{equation*}
        M_n(T) \leq \frac{2\sqrt{\pi}}{1 - e^{-1/10}} M_{[\sqrt{n}]}(A)
    \end{equation*}
    where $A=\psi(T)$.
\end{thm} 
\begin{proof}
Fix $n \in \mathbb{N}$. Then with Proposition \ref{sum_closed_form}, we have: \begin{align*}
    M_{[\sqrt{n}]}(A) &= {\frac{1}{[\sqrt{n}]}}\sum_{j = 0}^{[ \sqrt{n}]-1}\sum_{p = 0}^\infty \alpha_{p}^{(j)}T^p\\
    &\geq {\frac{1}{[\sqrt{n}]}}\sum_{j = 0}^{[ \sqrt{n}]-1}\sum_{p = 0}^{n-1} \alpha_{p}^{(j)}T^p\\
    &= \frac{1}{[ \sqrt{n}]}\left[I + \sum_{p=1}^{n-1} T^p \sum_{j = 0}^{[ \sqrt{n}] -1}\alpha_p^{(j)}  \right]\\
    &= \frac{1}{[ \sqrt{n}]}\left[I + \sum_{p=1}^{n-1}T^p \cdot2^{-([ \sqrt{n}] -1 + 2p)} \left[2^{[\sqrt{n}] -1}{2p \choose p} - {{[ \sqrt{n}] -1 + 2p }\choose p}\right]\right].
\end{align*}
We can rewrite the coefficient inside the sum as:
\begin{equation*}
    2^{-([ \sqrt{n}] -1 + 2p)} \left[2^{[\sqrt{n}] -1}{2p \choose p} - {{[ \sqrt{n}] -1 + 2p }\choose p}\right] = 4^{-p}\binom{2p}{p}\left(1-P_p^{(n)}\right).
\end{equation*}
With $P_{p}^{(n)}$ defined in Proposition \ref{product_estimate}. Propositions \ref{Hirschhorn_bound} and \ref{product_estimate} now give us:
\begin{align*}
    4^{-p}\binom{2p}{p}\left(1-P_p^{(n)}\right) \geq \frac{1}{\sqrt{\pi(p +1)}}\left(1 - e^{-1/10}\right)
\end{align*}
Hence: 
\begin{align*}
    M_{[\sqrt{n}]}(A)& \geq \frac{1}{[ \sqrt{n}]}\sum_{p=0}^{n-1}T^p \left(\frac{1}{\sqrt{\pi n}}\left[ 1-e^{-1/10} \right]\right)\\
    & \geq \frac{1}{n}\sum_{p=0}^{n-1}T^p \left(\frac{[\sqrt{n-1}]}{[\sqrt{n}]\sqrt{\pi}}\left[ 1-e^{-1/10} \right]\right)\\
    &\geq \frac{1-e^{-1/10}}{2\sqrt{\pi}}M_n(T).
\end{align*}
    Thus: \begin{equation}
        M_n(T) \leq \frac{2\sqrt{\pi}}{1-e^{-1/10}}M_{[ \sqrt{n}]}(A).\qedhere
    \end{equation}
\end{proof}

\section{Appendix}
Here we prove part (3) of Lemma \ref{ijk_lemma} and Lemma \ref{Upper_Lower_Kn}. We also provide an example of a mean-bounded non-positive matrix $T$ for which $A(T)$ is power-bounded, and give a complete proof of Theorem \ref{sup_2_bound} with an explicit upper bound.

\subsection{Proof of Lemma \ref{ijk_lemma} (3)} \begin{proof}
Let \begin{equation*}
    \kappa_p^{(n)}\equiv \frac{2n(n+p+1)-(n+1)(n+2p)}{(n+p+2)2(n+p)}
\end{equation*} We start with the following chain of equivalences: 
\begin{align*}
     \alpha_p^{(n)} - \alpha_p^{(n+1)} &\leq \alpha_{p+1}^{(n)}-\alpha_{p+1}^{(n+1)}\\
    \bigg{[}2^{-n-2p}{{n+2p}\choose{p}}\bigg{]}\kappa_p^{(n)}
    &\leq \bigg{[}2^{-n-2p-2}{{n+2p+2}\choose{p+1}}\bigg{]}\kappa_{p+1}^{(n)}\\
     4\frac{(n+2p)!}{p!(n+p)!} \frac{(p+1)!(n+p+1)!}{(n+2p+2)!} \kappa_p^{(n)}
    &\leq \kappa_{p+1}^{(n)}\\
4\frac{(p+1)(n+p+1)}{(n+2p+2)(n+2p+1)}\kappa_p^{(n)}
        &\leq \kappa_{p+1}^{(n)}\\
         4\frac{p+1}{n+2p+1}\frac{2n(n+p+1)-(n+1)(n+2p)}{2(n+2p)}
        &\leq 
\kappa_{p+1}^{(n)}(n+p+2) \\
        \frac{4(p+1)(n^2+n-2p)}{2(n+2p+1)(n+2p)} &\leq \frac{n^2+n-2p-2}{2(n+p+2)}\\
\end{align*}
 Upon clearing denominators and collecting terms by powers of $p$,  we find that the above inequalities are equivalent to:
        \begin{equation} \label{equiv}
            p(12n^2+12n-12)-12p^2 \leq n^4-2n^3-13n^2-10n.
        \end{equation}
Removing the $-12p^2$, we see that $p \leq K_n$ 
implies (\ref{equiv}), and this proves (3).
\end{proof}

\subsection{Proof of Lemma \ref{Upper_Lower_Kn}}
\begin{proof}
Fix $n\in\mathbb{N}$. We first define a more precise version of $K_n$ that provides a more detailed estimate. Consider the zeros of the following quadratic equation in $p$:
    \[f(p)=-12p^2+12(n^2+n-1)p-(n^4-2n^3-13n^2-10n).\]
    We have that $f(p)=0$ if and only if:
    \[p=-\frac{1}{24}\left(-12(n^2+n-1)\pm\sqrt{144(n^2+n-1)^2-48(n^4-2n^3-13n^2-10n)}\right).\]
    Note that we may choose $n$ large enough such that both roots are positive, and for the remainder of the proof we assume that both roots are nonnegative. Denote the larger root by $\overline{K_n}$ and the smaller root by $\underline{K_n}$. We consider several cases. Recall that $J_n = \frac{n^2+n}{2}$.
    \begin{enumerate}
        \item \emph{Case 1:} Suppose that $p < J_n$. Then  $p\leq\underline{K_n}$ if and only if:
    \[-24p+12(n^2+n-1)\geq\sqrt{144(n^2+n-1)^2-48(n^4-2n^3-13n^2-10n)},\]
    or, upon rearranging:
    \[-12p^2+12p(n^2+n-1)\leq n^4-2n^3-13n^2-10n,\]
    and it follows from the above that this inequality is equivalent to the inequality:
    \[\alpha_p^{(n)}-\alpha_p^{(n+1)}\leq\alpha_{p+1}^{(n)}-\alpha_{p+1}^{(n+1)}.\]
    \item \emph{Case 2:} The above computations show that $\underline{K_n}\leq p< J_n$ if and only if $\alpha_p^{(n)}-\alpha_p^{(n+1)}\geq\alpha_{p+1}^{(n)}-\alpha_{p+1}^{(n+1)}$.
    \item \emph{Case 3:} If $p>J_n$, then $p \leq \overline{K_n}$ iff: \[-24p +12(n^2+n-1) \geq - \sqrt{144(n^2+n-1)^2-48(n^4-2n^3-13n^2-10n)},\] Hence: \[-12p^2+12p(n^2+n-1) \geq n^4-2n^3-13n^2-10n,\]
    which again is equivalent to the inequality $\alpha_p^{(n)}-\alpha_p^{(n+1)}\geq\alpha_{p+1}^{(n)}-\alpha_{p+1}^{(n+1)}$.
    \item \emph{Case 4:} It follows from case 3 that $p \geq \overline{K_n}$ if and only if \begin{equation*}
        \alpha_p^{(n)}-\alpha_p^{(n+1)} \leq \alpha_{p+1}^{(n)}-\alpha_{p+1}^{(n+1)}
    \end{equation*}
    \item \emph{Case 5:} If $p = J_n$, it follows by definition that \begin{equation*}
        \alpha_p^{(n)}-\alpha_p^{(n+1)} \leq \alpha_{p+1}^{(n)}-\alpha_{p+1}^{(n+1)}.
    \end{equation*}
    \end{enumerate}
    This completes the proof.
\end{proof}
\subsection{Proof of Theorem \ref{sup_2_bound}}
\begin{proof}
When $n < 5$, $I_n < 0$, and so \begin{equation}
    \sum_{p=0}^\infty |\alpha_{p}^{(n)}-\alpha_{p+1}^{(n)}|(p+1) = \sum_{p=0}^\infty (\alpha_{p}^{(n)}-\alpha_{p+1}^{(n)})(p+1) =1.
\end{equation}
Fix $n \in \mathbb{N}, n \geq 5$ and $N \in \mathbb{N}$ such that $N > \lfloor I_n \rfloor$. An Abel summation yields: \begin{align*}
    \sum_{p=0}^N |\alpha_p^{(n)}-\alpha_{p+1}^{(n)}|(p+1) & =  \underbrace{(N+1) \sum_{p=0}^N |\alpha_{p}^{(n)}-\alpha_{p+1}^{(n)}|}_{(*)}- \underbrace{\sum_{p =0}^{N-1}\sum_{\ell = 0}^p |\alpha_{\ell}^{(n)} - \alpha_{\ell+1}^{(n)}|}_{(**)}. 
\end{align*}
For the first sum, we have: \begin{align*}
    (*)&=  (N+1)\left( \sum_{p=0}^{\lfloor I_n\rfloor}\alpha_{p+1}^{(n)}- \alpha_{p}^{(n)} + \sum_{p=\lfloor I_n\rfloor +1}^{N} \alpha_{p}^{(n)}-\alpha_{p+1}^{(n)}\right)\\
    &= (N+1)\left(2\alpha^{(n)}_{\lfloor I_n\rfloor + 1}- \alpha_0^{(n)} - \alpha_{N+1}^{(n)}\right). 
\end{align*}
For the second sum, we similarly have: 
\begin{align*}
    (*) & = \sum_{p=0}^{\lfloor I_n \rfloor} \sum_{\ell=0}^p \alpha^{(n)}_{\ell+1}-\alpha^{(n)}_{\ell} + \sum_{p=\lfloor I_n\rfloor + 1}^{N-1}\left(\sum_{\ell = 0}^{\lfloor I_n\rfloor}\alpha_{\ell+1}^{(n)} - \alpha_{\ell}^{(n)} + \sum_{\ell = \lfloor I_n \rfloor + 1}^p \alpha^{(n)}_{\ell} - \alpha^{(n)}_{\ell +1} \right)\\
    &= \sum_{p=0}^{\lfloor I_n \rfloor} \left(\alpha_{p+1}^{(n)}-\alpha_0^{(n)}\right) + \sum_{p=\lfloor I_n\rfloor + 1}^{N-1}\left(2\alpha^{(n)}_{\lfloor I_n\rfloor +1} - \alpha_0^{(n)} -\alpha^{(n)}_{p+1}\right)\\
    &=  \sum_{p=0}^{\lfloor I_n \rfloor} \alpha_{p+1}^{(n)} - \sum_{p=\lfloor I_n \rfloor + 1}^{N-1} \alpha_{p+1}^{(n)} - N\alpha_0^{(n)} + (N-\lfloor I_n \rfloor -1)\left(2\alpha_{\lfloor I_n\rfloor + 1}^{(n)}\right ).
\end{align*}
Adding and subtracting $\sum_{p = \floor{I_n} + 1}^{N-1} \al{p}{n}$, we have
\begin{align*}
    (*) &= \sum_{p=0}^{N-1}\alpha_{p+1}^{(n)} - 2\sum_{p = \floor{I_n} +1}^{N-1} \alpha_{p + 1}^{(n)} - N\alpha_0^{(n)} + (N-\lfloor I_n \rfloor -1)\left(2\alpha_{\lfloor I_n\rfloor + 1}^{(n)}\right )\\
   & = \sum_{p=0}^{N}\alpha_{p}^{(n)} - 2\sum_{p = \floor{I_n}+1}^{N} \alpha_{p}^{(n)} - (N+1)\alpha_0^{(n)}   + ((N+1)-\lfloor I_n \rfloor -1)\left(2\alpha_{\lfloor I_n\rfloor + 1}^{(n)}\right ).
\end{align*}

Combining these terms yields:
\begin{align*}
    \sum_{p=0}^N |\alpha_p^{(n)}-\alpha_{p+1}^{(n)}|(p+1) & =  2\sum_{p = \floor{I_n} +1}^{N}\alpha_p^{(n)}- \sum_{p =0}^N\alpha_{p}^{(n)}+2(\floor{I_n}+1)\al{\floor{I_n}+1}{n} \\
    & \hspace{0.5in}- (N+1)\al{N+1}{n}. \\
\end{align*}
Letting $N \to \infty$, we have $(N+1) \alpha_{N + 1}^{(n)} \to 0$ by Corollary \ref{N_alpha_N_estimate}, and hence we have: 
\begin{align*}
\sum_{p=0}^\infty |\alpha_p^{(n)}-\alpha_{p+1}^{(n)}|(p+1) &\leq 1 + 2(\lfloor I_n \rfloor +1)\alpha^{(n)}_{\lfloor I_n \rfloor + 1}\\
&\leq 1 + 2\frac{n(\lfloor I_n \rfloor + 1) }{n + 2(\lfloor I_n \rfloor+1)}\frac{1}{\sqrt{\pi (\lfloor I_n \rfloor +1)}}\\
&\leq 1 +  \frac{n}{\sqrt{\pi(\lfloor I_n \rfloor + 1)}}\\
&\leq 1 + \frac{5}{\sqrt{\pi}}
\end{align*}
Hence \begin{equation}
    \sup_{n \in \mathbb{N}}\sum_{p =0}^\infty |\alpha_{p}^{(n)}- \alpha_{p+1}^{(n)}|(p+1) < 1 + \frac{5}{\sqrt{\pi}} < \infty. 
\end{equation}
\end{proof}

\subsection{Nonpositive Mean-Bounded $T$ with $A(T)$ Power-Bounded}
\begin{prop}
    Define $T^n:\mathbb{R}^2\to\mathbb{R}^2$ for any $n\in\mathbb{N}$ via the following matrix representation:
    \[T^n=\begin{pmatrix}
    (-1)^n & 2n(-1)^{n+1} \\
    0 & (-1)^n \\
    \end{pmatrix}.\]
    Then we have:
    \[A^n(T)=\begin{pmatrix}
    (\sqrt2-1)^n & -2^{1-n}\cdot n\cdot S(n) \\
    0 & (\sqrt2-1)^n \\
    \end{pmatrix},\]
    where we define:
    \[S(n)=\sum_{p=0}^\infty\left(-\frac{1}{4}\right)^p\binom{n+2p-1}{p-1}.\]
\end{prop}
\begin{proof}
    Note that the series:
    \[\sum_{p=0}^\infty(-1)^p\alpha_{p}^n=\sum_{p=0}^\infty\frac{(-1)^p\cdot n}{n+p}2^{-(n+2p)}\binom{n+2p-1}{p}\]
    corresponds to evaluation of $\psi^n(x)$ at $x=-1$, which has the value $(\sqrt2-1)^n$ for any $n\in\mathbb{N}$. In addition, we have:
    \begin{align*}
        -\sum_{p=0}^\infty\frac{2p\cdot(-1)^p\cdot n}{n+p}4^{-p}\binom{n+2p-1}{p}&=-2n\sum_{p=0}^\infty\frac{p\cdot(-1)^p}{n+p}4^{-p}\cdot\frac{(n+2p-1)!}{p!(n+p-1)!} \\
        &=-2n\sum_{p=0}^\infty\left(-\frac{1}{4}\right)^p\frac{(n+2p-1)!}{(p-1)!(n+p)!} \\
        &=-2n\sum_{p=0}^\infty\left(-\frac{1}{4}\right)^p\binom{n+2p-1}{p-1}.
    \end{align*}
    Upon multiplying both sides by $2^{-n}$, this shows that the desired matrix representation of $A^n(T)$ holds.
\end{proof}
\begin{remark}
    The matrix $T$ in this proposition was originally given by I. Assani as a counterexample to the mean ergodic theorem when the operator $T$ is not assumed to be positive (see \cite{Emilion}, page 643).
\end{remark}

\section*{Acknowledgements}
The authors would like to thank Y. Tomilov for bringing to their attention the previous work done in \cite{Dungey} and \cite{Tomilov} regarding Ritt operators. They would also like to thank the previous reviewer for their valuable suggestions as well as for bringing the authors' attention to Lootgieter's paper \cite{Lootgieter}.

\nocite{*}
\bibliographystyle{alpha}

\end{document}